\documentclass[12pt]{amsart}

\usepackage{amsfonts,amssymb,enumerate,bm,hhline,makecell,array,mathtools}
\usepackage{mathrsfs}
\usepackage{comment}
\usepackage[normalem]{ulem}
\usepackage[colorlinks=true,citecolor=blue, urlcolor=blue, linkcolor=blue, %pagebackref=true
]{hyperref}
\usepackage{amscd}
\usepackage[shortlabels]{enumitem}
\setlist[itemize]{noitemsep,nolistsep,topsep=-3pt}
\setlist[enumerate]{noitemsep,nolistsep,topsep=-3pt}
\usepackage{tensor}

\usepackage{tikz}
\usetikzlibrary{calc,trees,arrows,chains,shapes.geometric,%
	decorations.pathreplacing,decorations.pathmorphing,shapes,%
	matrix,shapes.symbols}
\tikzset{
	>=stealth',
	punktchain/.style={
		rectangle,
		rounded corners,
		draw=black, thick,
		minimum height=3em,
		text centered,
		on chain},
	line/.style={draw, thick, <-},
	element/.style={
		tape,
		top color=white,
		bottom color=blue!50!black!60!,
		minimum width=8em,
		draw=blue!40!black!90, very thick,
		text width=10em,
		minimum height=3.5em,
		text centered,
		on chain},
	every join/.style={->, thick,shorten >=1pt},
	decoration={brace},
	tuborg/.style={decorate},
	tubnode/.style={midway, right=2pt},
}
\usepackage{tikz-cd}
\usetikzlibrary{calc,matrix,patterns,shadings,backgrounds,fit}
\pgfdeclarelayer{myback}
\pgfsetlayers{myback,background,main}

\usepackage{extarrows}
\usepackage[new]{old-arrows}
\usepackage[toc,page]{appendix}
\usepackage[all,ps,cmtip,rotate]{xy}
\usepackage{color}
\allowdisplaybreaks

\def\fM{\mathfrak{m}}

\def\cD{\mathscr{D}}
\def\cA{\mathscr{A}}

\def\cF{\mathscr{F}}

\def\cO{\mathscr{O}}

\def\cP{\mathscr{P}}

\def\cJ{\mathscr{J}}

\def\cM{\mathscr{M}}

\def\cT{\mathscr{T}}

\def\cU{\mathscr{U}}

\def\cW{\mathscr{W}}
\def\cX{\mathscr{X}}
\def\cY{\mathscr{Y}}
\def\cZ{\mathscr{Z}}

\def\C{\ensuremath{\mathbb{C}}}

\def\H{\ensuremath{\mathbb{H}}}

\def\P{\ensuremath{\mathbb{P}}}
\def\Q{\ensuremath{\mathbb{Q}}}
\def\R{\ensuremath{\mathbb{R}}}
\def\Z{\ensuremath{\mathbb{Z}}}

\def\bv{\ensuremath{\mathbf{v}}}

\DeclareMathOperator{\Aut}{Aut}

\DeclareMathOperator{\Br}{Br}
\DeclareMathOperator{\ch}{ch}

\DeclareMathOperator{\coh}{coh}

\DeclareMathOperator{\codim}{codim}

\DeclareMathOperator{\Def}{Def}

\newcommand{\Db}{\mathrm{D^{b}}}
\newcommand{\Dqc}{\mathrm{D_{qc}}}
\newcommand{\Dqcinf}{\mathcal{D}_\mathrm{qc}}
\newcommand{\Dperf}{\mathrm{D_{perf}}}
\newcommand{\Dperfinf}{\mathcal{D}_\mathrm{perf}}

\DeclareMathOperator{\End}{End}

\DeclareMathOperator{\Ext}{Ext}
\DeclareMathOperator{\ext}{ext}

\DeclareMathOperator{\HH}{HH}

\DeclareMathOperator{\Hom}{Hom}

\def\Im{\mathop{\rm Im}\nolimits}

\DeclareMathOperator{\NS}{NS}

\DeclareMathOperator{\Pic}{Pic}

\DeclareMathOperator{\rk}{rk}

\DeclareMathOperator{\Spec}{Spec}
\DeclareMathOperator{\Stab}{Stab}

\DeclareMathOperator{\Sym}{Sym}

\DeclareMathOperator{\Tr}{Tr}
\DeclareMathOperator{\td}{td}

\DeclareMathOperator{\id}{\mathrm{id}}
\DeclareMathOperator{\Ho}{\mathrm{h}\!}

\theoremstyle{plain}
\newtheorem{lemm}{Lemma}[section]
\newtheorem{theo}[lemm]{Theorem}
\newtheorem{coro}[lemm]{Corollary}
\newtheorem{prop}[lemm]{Proposition}

\newtheorem*{claim*}{Claim}
\newtheorem{conj}[lemm]{Conjecture}

\theoremstyle{definition}
\newtheorem{defi}[lemm]{Definition}
\newtheorem{rema}[lemm]{Remark}

\newtheorem{exam}[lemm]{Example}
\newtheorem{ques}[lemm]{Question}

\makeatletter
\@namedef{subjclassname@2020}{%
  \textup{2020} Mathematics Subject Classification}
\makeatother

\def\citestacks#1{\cite[\href{https://stacks.math.columbia.edu/tag/#1}{Tag #1}]{stacks-project}}

\title[Hyper-K\"ahler varieties]{Hyper-K\"ahler varieties: Lagrangian fibrations, atomic sheaves, and categories}

\begin{document}

\author[A.~Bottini]{Alessio Bottini}
\address{\parbox{0.9\textwidth}{Mathematisches Institut, Universit\"at Bonn\\[1pt]
Endenicher Allee 60, 53115 Bonn, Germany
\vspace{1mm}}}
\email{bottini@math.uni-bonn.de}
\urladdr{\url{https://alessiobottini.github.io/}}

\author[E.~Macr\`i]{Emanuele Macr\`i}
\address{\parbox{0.9\textwidth}{Universit\'e Paris-Saclay, CNRS, Laboratoire de Math\'ematiques d'Orsay\\[1pt]
Rue Michel Magat, B\^at. 307, 91405 Orsay, France
\vspace{1mm}}}
\email{{emanuele.macri@universite-paris-saclay.fr}}
\urladdr{\url{https://www.imo.universite-paris-saclay.fr/~macri/}}

\author[P.~Stellari]{Paolo Stellari}
\address{\parbox{0.9\textwidth}{Dipartimento di Matematica ``F.~Enriques'', Universit\`a degli Studi di Milano\\[1pt]
Via Cesare Saldini 50, 20133 Milano, Italy
\vspace{1mm}}}
\email{paolo.stellari@unimi.it}
\urladdr{\url{https://sites.unimi.it/stellari}}

\thanks{A.B.~and E.M.~were partially supported by the ERC Synergy grant ERC-2020-SyG-854361-HyperK. P.S.~ was partially supported by the ERC Consolidator grant ERC-2017-CoG-771507-StabCondEn, by the research project PRIN 2017 `Moduli and Lie Theory', and by the research project FARE 2018 HighCaSt (grant number R18YA3ESPJ)}

\begin{abstract}
We review recent developments in the theory of compact hyper-K\"ahler varieties, from the viewpoint of Lagrangian fibrations, moduli spaces of stable sheaves, and derived categories.
These notes originated from the lecture by the second named author at the \emph{2025 Summer Institute in Algebraic Geometry}, Colorado State University, Fort Collins (USA), July 14 -- August 1, 2025.
\end{abstract}

\maketitle

\setcounter{tocdepth}{1}
\tableofcontents

%%%%%%%%%%%%%%%%%%

\section{Introduction}\label{sec:intro}

Hyper-K\"ahler manifolds are compact K\"ahler manifolds, which are simply connected and have a unique holomorphic symplectic form up to constant.
They are the higher dimensional analogue of K3 surfaces: the key property of hyper-K\"ahler manifolds is that their geometry is encoded in lattice and Hodge theoretical data coming essentially from their second cohomology group. As such, they can be studied with a mix of techniques, coming from complex geometry, moduli spaces, arithmetic, and category theory.

The goal of this note is to present some recent developments in the theory of projective hyper-K\"ahler manifolds from the point of view of Lagrangian fibrations, namely integrable systems in the holomorphic setting. 
More in detail, by the Beauville--Bogomolov Decomposition Theorem, hyper-K\"ahler manifolds form one of the three building blocks, together with abelian varieties and strict Calabi--Yau manifolds, in the classification of smooth projective varieties with trivial canonical bundle. Thus it would be helpful yet ambitious to have a classification of projective hyper-K\"ahler manifolds.
The structure of a Lagrangian fibration helps towards the classification in two directions. First of all it turns out to be useful in constructing examples, by compactifying smooth abelian fibrations with special properties. 
And second of all, it helps in restricting the topological type.
The main conjecture in the theory is indeed that every hyper-K\"ahler manifold can be deformed to one admitting a Lagrangian fibration.

An important technique to understand and study Lagrangian fibrations is the theory of moduli spaces of stable sheaves. Pioneered in works by Mukai on K3 surfaces, the theory recently developed also in the higher dimensional case, with the concept of atomic sheaf.
Realizing a hyper-K\"ahler manifold as a moduli space over a smaller dimensional hyper-K\"ahler (or better K-trivial) manifold has several advantages and applications, including the study of cycles and Hodge classes, singularities, ample cones, but also to construct Lagrangian fibrations.

Finally, the study of moduli spaces is naturally related to (triangulated or higher) categories.
In fact, one of the fundamental and classical techniques which are used to study the geometry of moduli spaces is the Fourier--Mukai transform.
This comes handy to construct examples and to show nonemptiness of moduli spaces, which is the key to many applications, including to construct Lagrangian fibrations. 
An observation here is that, in order to get a satisfactory geometric picture, one is forced to leave the `commutative world' of the bounded derived categories of actual projective varieties and delve into categorical noncommutative analogues of such varieties. 
In this generality, we will need to consider two additional structures on those categories: their higher categorical structure and a suitable notion of stability.

Our aim is to present these three developments in the theory of hyper-K\"ahler manifold, Lagrangian fibrations, atomic sheaves, and categories, with a view towards classification of hyper-K\"ahler manifolds.
These will be the content of Sections~\ref{sec:Lagrangian},~\ref{sec:atomic}, and~\ref{sec:categories}, respectively.

A second aspect of our presentation is that, from the viewpoint of the classification of algebraic varieties, we must allow some singularities as well. 
The incarnation of this basic idea in the theory of hyper-K\"ahler manifolds is a fairly recent research topic, which has now reached a complete maturity, at a similar level as the smooth case.
Hence, in Section~\ref{sec:HK}, we will review the basic definitions and main general results on hyper-K\"ahler varieties (like the Torelli Theorem and the Beauville--Bogomolov Decomposition Theorem) in the singular setting.
It is also important to consider deformations of hyper-K\"ahler varieties in the analytic context.
We will mostly present the theory in the projective setting (since also moduli spaces and categories work best for projective hyper-K\"ahler varieties), but we will point out results in the analytic setting throughout the paper, when needed.

The reader should be warned that this paper should not be considered as a comprehensive presentation of the existing literature: it is just a guided tour through some of the major recent results, which will naturally lead us towards the presentation of our contributions in this area.

We work over the complex numbers.
When considering objects or notions of topological nature, like fundamental groups and singular cohomology, or being simply connected, we will always mean with respect to the analytic topology.

\subsection*{Acknowledgements}
It is a great pleasure for us to thank Enrico Arbarello, Arend Bayer, Aaron Bertram, Alberto Canonaco, Olivier Debarre, Daniel Huybrechts, Alexander Kuznetsov, Chunyi Li, Mirko Mauri, Giovanni Mongardi, Amnon Neeman, Kieran O'Grady, Alexander Perry, Laura Pertusi, Antonio Rapagnetta, Giulia Sacc\`a, Claire Voisin, and Xiaolei Zhao for many discussions through the years. 
Daniel Huybrechts, Mirko Mauri, Kieran O'Grady, Alexander Perry, and Giulia Sacc\`a also provided valuable comments on a preliminary version of this paper which helped us improving it.
We are also very grateful to Renzo Cavalieri, Gavril Farkas, Angela Gibney, Christopher Hacon, Andrei Okounkov, Bjorn Poonen, Karen Smith, and Chenyang Xu for organizing the 2025 Summer Research Institute in Algebraic Geometry and making it such an inspiring event.

%%%%%%%%%%%%%%%%%%

\section{Hyper-K\"ahler varieties}\label{sec:HK}

In this section we recall and discuss the definition of hyper-K\"ahler variety, and present the first examples. 
The various definitions and comparison between them are in Section~\ref{subsec:Def}; in Section~\ref{subsec:BBdec} we will review the Beauville--Bogomolov Decomposition Theorem and in Section~\ref{subsec:Torelli} the Torelli Theorem.
Finally, examples, mostly coming from moduli spaces of sheaves (or complexes) on K3 surfaces, are in Section~\ref{subsec:Examples}: the reader willing to get a grasp on the definitions can quickly move to this section.

In what follows, all varieties will be integral, normal, and quasi-projective over $\C$. 
For such a variety $X$, we denote by $X_\mathrm{reg}\subset X$ the open subset of regular points, and by $X_\mathrm{sing}$ its complement.

We refer to~\cite{Huy:HKbook,Deb:HK} for a more in depth presentation of the theory.

\subsection{Definition}\label{subsec:Def}

The definition of irreducible symplectic variety appears in the recent literature in different variants. Throughout this paper we adopt the following.

\begin{defi}\label{def:HK}
Let $M$ be an integral projective variety over $\C$, with terminal $\Q$-factorial singularities.
We say that $M$ is a  {\sf projective hyper-K\"ahler variety} if
\begin{enumerate}[{\rm (i)}]
    \item $H^0(M_\mathrm{reg},\Omega^2_{M_\mathrm{reg}})\cong \C\cdot\sigma$, and $\sigma$ is nondegenerate;
    \item $M_\mathrm{reg}$ is simply connected.
\end{enumerate}
\end{defi}

When $M$ is smooth, this is the standard definition of \emph{projective hyper-K\"ahler manifold} (see, e.g.,~\cite{Calabi,Hitchin,Bea:c10,Huy:Basic}).
We should stress instead that in the singular case this is not a standard definition: we start by comparing this definition with others in the literature, before going into the basic properties of hyper-K\"ahler varieties (in Section~\ref{subsec:Torelli} and Section~\ref{subsec:BBdec}).
This part is technical and it can be skipped on first reading.
More details about these definitions are in~\cite[Section 8.B.4]{GKP:Singular},~\cite[Section 2]{Sch:Fujiki},~\cite[Section 3]{BL:GlobalModuli}, and~\cite[Section 2.1]{EFGMS:Boundedness}.

The definition of symplectic variety is due to Beauville~\cite{Bea:Symplectic}:

\begin{defi}\label{def:symplectic}
Let $M$ be an integral, normal, quasi-projective variety over $\C$ and let $\sigma\in H^0(M_\mathrm{reg},\Omega^2_{M_\mathrm{reg}})$ be a 2-form on $M_\mathrm{reg}$.
We say that $(M,\sigma)$ is {\sf symplectic} if $\sigma$ is closed, nondegenerate, and it extends to a regular 2-form on any resolution of singularities of $M$.
\end{defi}

A few comments on Definition~\ref{def:symplectic} are in order here.
In this paper, by \emph{resolution of singularities} we mean a proper birational morphism $f\colon\widetilde{M}\to M$ such that $\widetilde{M}$ is regular.
The condition on the extension of $\sigma$ on any resolution means that there exists a $2$-form $\widetilde{\sigma}\in H^0\big(\widetilde{M},\Omega^2_{\widetilde{M}}\big)$ such that 
\[
\widetilde{\sigma}|_{f^{-1}(M_\mathrm{reg})}=f|_{f^{-1}(M_\mathrm{reg})}^*\sigma
\]
Since any two resolutions can be dominated by a common one, the extension condition can be checked on a single resolution.
Notice that the 2-form $\widetilde{\sigma}$ need not to be nondegenerate.

An immediate consequence of Definition~\ref{def:symplectic} (see, e.g.,~\cite[Proposition 1.3]{Bea:Symplectic}) is that $M$ has rational Gorenstein singularities.
By~\cite[Corollary 1.8]{KS:Extension}, a converse holds: if $M$ has rational singularities, then every holomorphic form defined on $M_\mathrm{reg}$ extends uniquely to a holomorphic form on any resolution $\widetilde{M}\to M$.
Moreover, if $M$ is projective with rational singularities then, by~\cite[Theorem 1.13]{KS:Extension}, any holomorphic form on $M_\mathrm{reg}$ is closed.
Since, in the context of Definition~\ref{def:HK}, $M$ is projective with terminal (hence rational) singularities, this proves the following result.

\begin{lemm}\label{lem:HKsymplectic}
Let $M$ be a projective hyper-K\"ahler variety.
Then $(M,\sigma)$ is symplectic.
\end{lemm}

%Moreover, it has Gorenstein singularities.
%In this case, having canonical singularities is actually equivalent to have rational singularities (see~\cite[Th\'eor\`eme 1]{Elk:Canonical}).

Let us recall that on a normal quasi-projective variety $M$, if we denote by $j\colon M_\mathrm{reg}\hookrightarrow M$ the open inclusion, then we have
\[
j_*\Omega^p_{M_\mathrm{reg}} \cong \Omega^{[p]}_M,
\]
for all $0\leq p\leq\dim(M)$, where $\Omega^{[p]}_M\coloneqq (\Lambda^p\,\Omega^1_M)^{\vee\vee}$ denotes the sheaf of \emph{reflexive differentials} on $M$ (see~\cite[Section 2.E]{GKKP:Differential}).
Moreover, by~\cite[Theorem 2.4]{GKP:Singular}, if $M$ has klt singularities, we have
\[
\Omega^{[p]}_M \cong f_*\Omega^p_{\widetilde{M}},
\]
for all resolutions of singularities $f\colon\widetilde{M}\to M$.
From this we deduce that, if $M$ is also projective,
\begin{equation}\label{eq:HodgeDuality}
H^0(M_\mathrm{reg},\Omega^p_{M_\mathrm{reg}})=H^0(M,\Omega_M^{[p]})\cong\overline{H^p(M,\cO_M)}.    
\end{equation}
for all $0\leq p\leq\dim(M)$ (see~\cite[Proposition 6.9]{GKP:Singular}).
By~\cite[Theorem 5.2]{Keb:Pullback}, we can pull-back reflexive differentials along surjective morphisms of normal quasi-projective varieties, as long as the base has klt singularities, in a way that is compatible with the usual pull-back of K\"ahler differentials on the regular locus.

Let us also recall that a morphism of integral, normal, quasi-projective varieties is a \emph{quasi-\'etale cover} if it is finite, surjective, and \'etale in codimension~1.
By the Zariski--Nagata Purity Lemma~\citestacks{0BMB}, for an integral normal quasi-projective variety $X$, there is an equivalence of categories between quasi-\'etale covers of $X$ and finite \'etale covers of $X_\mathrm{reg}$ (and thus with the category of finite sets with a transitive group action of the fundamental group of $X_\mathrm{reg}$).

\begin{defi}\label{def:MoreSymplectic}
Let $(M,\sigma)$ be a projective symplectic variety.
We say that $M$ is:
\begin{enumerate}[{\rm (i)}]
    \item\label{enum:MoreSymplectic1} {\sf primitive symplectic}, if $H^0(M,\Omega_M^{[1]})=0$ and $H^0(M,\Omega_M^{[2]})\cong \C\cdot\sigma$;
    \item\label{enum:MoreSymplectic2} {\sf irreducible symplectic}, if all quasi-\'etale covers $M'\to M$ are primitive symplectic.
\end{enumerate}
\end{defi}

Clearly, an irreducible symplectic variety is primitive symplectic.
The definition of irreducible symplectic variety is different from the one in~\cite[Definition 8.16]{GKP:Singular}, but equivalent (as a consequence of the Beauville--Bogomolov Decomposition Theorem; see Section~\ref{subsec:BBdec}).

The following results completes Lemma~\ref{lem:HKsymplectic}.

\begin{prop}\label{prop:HKisIHS}
Let $M$ be a projective hyper-K\"ahler variety.
Then $M$ is irreducible symplectic.
\end{prop}

\begin{proof}
First of all, by Lemma~\ref{lem:HKsymplectic}, we know that $(M,\sigma)$ is symplectic.
Since $M$ is a normal variety, we have a surjective morphism between the fundamental groups of $M_\mathrm{reg}$ and the one of $M$ (see~\cite[0.7(B), p.~33]{FL:Connectivity}).
Since in our assumption $M_\mathrm{reg}$ is simply connected, we deduce that $M$ is also simply connected.
Since $M$ has terminal (hence klt) singularities, its fundamental group is isomorphic to the fundamental group of any resolution of singularities $\widetilde{M}\to M$ (see~\cite[Theorem 1.1]{Tak:FundamentalGroup}); in particular, $\widetilde{M}$ is also simply connected. 
Since $M$ (hence $\widetilde{M}$) is projective, we deduce that $H^1(\widetilde{M},\cO_{\widetilde{M}})=0$.
Since $M$ has rational singularities, we have then
\[
H^1(M,\cO_M)=H^1(\widetilde{M},\cO_{\widetilde{M}})=0.
\]
By~\eqref{eq:HodgeDuality}, we conclude that $H^0\big(M,\Omega_M^{[1]}\big)=0$, and so $M$ is primitive symplectic.
Finally, since $M_\mathrm{reg}$ is simply connected, there is no nontrivial quasi-\'etale cover, hence $M$ is irreducible symplectic, as we wanted.
\end{proof}

An important result in the theory is~\cite[Corollary 13.2]{GGK:klt}: the algebraic fundamental group of the regular locus of an irreducible symplectic variety is finite.
Hence, given a irreducible symplectic variety $M$, there exists a quasi-\'etale cover $M'\to M$ (the \emph{universal quasi-\'etale cover}) which is again irreducible symplectic and $M_\mathrm{reg}'$ is algebraically simply connected.
This is called a \emph{strict irreducible symplectic variety} in~\cite{EFGMS:Boundedness}.
We mention the following general expectation, which we learned from Rapagnetta (see~\cite[Sections 1.4 \& 13]{GGK:klt}, \cite[Remark 1.11]{PR:K3}, and \cite[Conjecture 1.3]{Wang}):

\begin{conj}\label{conj:Rapagnetta}
Let $M$ be a projective irreducible symplectic variety.
Then $M_\mathrm{reg}$ has finite fundamental group.
\end{conj}

If Conjecture~\ref{conj:Rapagnetta} is true, then $M_\mathrm{reg}'$ is actually simply connected.
We will briefly review in Section~\ref{subsec:EFGMS} that Conjecture~\ref{conj:Rapagnetta} follows from a general conjecture on Lagrangian fibrations for irreducible symplectic varieties.

As a final comment, we observe the following converse to Proposition~\ref{prop:HKisIHS}:

\begin{prop}\label{prop:IHSisHK}
Let $M$ be a projective primitive symplectic variety such that its regular locus $M_\mathrm{reg}$ is simply connected.
Then any $\Q$-factorial terminalization is a projective hyper-K\"ahler variety.
\end{prop}

We recall that, a \emph{$\Q$-factorial terminalization} of $M$ is a proper birational morphism $g\colon\widehat{M}\to M$ where $\widehat{M}$ is a projective variety with $\Q$-factorial terminal singularities such that $\omega_{\widehat{M}}=g^*\omega_M(\cong\cO_{\widehat{M}})$.
By~\cite[Corollary 1.4.3]{BCHM}, $\Q$-factorial terminalizations always exist.

\begin{proof}
Let $\widehat{M}\to M$ be a $\Q$-factorial terminalization.
By~\cite[Proposition 11]{Sch:Fujiki}, $\widehat{M}$ is again primitive symplectic.
Moreover, $\widehat{M}_\mathrm{reg}$ is again simply connected.
Indeed, $g$ induces an isomorphism $g^{-1}(M_\mathrm{reg})\xrightarrow{\cong} M_\mathrm{reg}$.
Then, by applying~\cite[0.7(B), p.~33]{FL:Connectivity} to the inclusion $g^{-1}(M_\mathrm{reg})\subset\widehat{M}_\mathrm{reg}$, we deduce that $\widehat{M}_\mathrm{reg}$ is simply connected as well.
\end{proof}

\subsection{Primitive symplectic varieties and the Torelli Theorem}\label{subsec:Torelli}

The notion of primitive symplectic variety in Definition~\ref{def:MoreSymplectic}\ref{enum:MoreSymplectic1} has been singled out in works by Schwald~\cite{Sch:Fujiki}\footnote{In \emph{loc.~cit.}, these are unfortunately called irreducible symplectic varieties.} and Bakker--Lehn~\cite{BL:GlobalModuli}.
It is the correct notion for proving the Torelli Theorem and studying Lagrangian fibrations.
We will briefly review deformations and the Torelli Theorem in this section.
The special case of hyper-K\"ahler varieties is discussed at the end.

The first remark is that, as in the smooth case, the definition can be extended to the analytic setting, where `projective' is replaced by `compact K\"ahler analytic space' (see~\cite[Section 2.3 \& Definition 3.1]{BL:GlobalModuli}).
We will not deal with the analytic case in this paper in detail, but it is necessary for many of the results we state in this and the next section (and in their proofs).

The second observation is that, since a symplectic variety $M$ has rational singularities, (the torsion-free part of) its second cohomology group $H^2(M,\Z)_\mathrm{tf}$ carries a pure weight~2 Hodge structure (see~\cite[Corollary 3.5]{BL:GlobalModuli}).
Moreover, we have the identifications
\[
H^{2,0}(M)=H^0\big(M,\Omega_M^{[2]}\big)\qquad H^{0,2}=H^2(M,\cO_M)
\]
and the Hodge duality~\eqref{eq:HodgeDuality}.

Let $M$ denote a compact K\"ahler primitive symplectic analytic variety (for simplicity we will simply refer in the sequel to a primitive symplectic variety, and stress when it is projective in the algebraic setting) of dimension~$2n$.
We first discuss deformations.
We recall that a deformation in our context is just a proper and flat morphism, and a deformation is \emph{locally trivial} if every point has an analytic neighborhood which is a trivial deformation (see, e.g.,~\cite[Definition 4.1]{BL:GlobalModuli}).
In particular, this condition guarantees that the singularities stay the same.
A key point is~\cite[Theorem 4.7]{BL:GlobalModuli} (extending to the singular case the Bogomolov--Tian--Todorov Theorem): if we denote by $\mathrm{Def}^\mathrm{lt}(M)$ the space of locally trivial deformations of $M$, then $\mathrm{Def}^\mathrm{lt}(M)$ is universal and smooth of dimension $h^{1,1}(M)$.
Moreover, by \cite[Corollary 4.11]{BL:GlobalModuli}, a small locally trivial deformation of a primitive symplectic variety is again primitive symplectic.
We say that two primitive symplectic varieties are {\sf deformation-equivalent} if they can be connected by a locally trivial deformation.

As in the smooth case, having a good theory of deformations has important consequences. The first one is that there is an integral, indivisible, nondegenerate quadratic form $q_M$ on $H^2(M,\Z)_\mathrm{tf}$, called the {\sf Beauville--Bogomolov--Fujiki form} (BBF form, for short), with signature $(3,b_2(M)-3)$, positive on ample/K\"ahler classes, compatible with the Hodge structure, and it is invariant under locally trivial deformations.
The BBF form generalizes the usual intersection form on surfaces.
In the smooth case, the existence and its basic properties were proved in~\cite{Bea:c10,Fuj:dR}, and the extension to the singular case is in~\cite{Nam:Extension,Kir:Period,Mat:Base,Sch:Fujiki,BL:GlobalModuli}; in particular, we refer to~\cite[Section 5.1]{BL:GlobalModuli} for all the above properties and for the local Torelli Theorem which is needed to prove them (see also~\cite{Nam:Def1,Nam:Def2}).
We will not recall here the definition of the BBF form, but we simply observe that it is characterized (up to sign) by the following property: there exists a constant $c_M\in\Q_{>0}$, called the {\sf Fujiki constant}, such that the {\sf Fujiki relation} holds
\[
\int_M \alpha^{2n} = (2n-1)!!\, c_M\, q_M(\alpha)^n,
\]
for all $\alpha\in H^2(M,\Z)$, where $(2n-1)!!$ denotes the double factorial.

This leads to a projectivity criterion: a primitive symplectic variety $M$ is projective if and only if there is a (1,1)-class in $\alpha\in H^2(M,\Q)$ such that $q_M(\alpha)>0$ (see~\cite[Theorem 1.2]{BL:GlobalModuli}; in the smooth case this is~\cite[Theorem 3.11]{Huy:Basic}).
In particular, every primitive symplectic analytic variety is deformation-equivalent to a projective primitive symplectic variety.

The second consequence is that we can glue the local pictures.
More precisely, let us fix a lattice $(\Lambda,q)$ or rank $b$ and signature $(3,b-3)$.
We can then speak about the moduli space $\mathfrak{M}_{\Lambda}$ of {\sf marked primitive symplectic varieties} $(M,\eta)$ (with respect to locally trivial deformations), where $\eta\colon(H^2(M,\Z)_\mathrm{tf},q_M)\xrightarrow{\cong}(\Lambda,q)$ is an isometry, which is then a not-necessarily-Hausdorff complex manifold of dimension~$b-2$.
Notice that $O(\Lambda)$ naturally acts on $\mathfrak{M}_{\Lambda}$.
We then have a period map
\[
P_\Lambda\colon\mathfrak{M}_{\Lambda}\longrightarrow\Omega_\Lambda\qquad (M,\eta)\longmapsto\eta(H^{2,0}(M)),
\]
where the period domain is defined by
\[
\Omega_\Lambda\coloneqq\left\{[x]\in\P(\Lambda_\C)\,:\, q(x)=0\text{ and }q(x,\overline{x})>0\right\}.
\]

The Torelli Theorem (see~\cite[Theorem 1.1]{BL:GlobalModuli}; in the smooth case this is~\cite{Ver:Torelli,Huy:Torelli}; see also~\cite{Menet:Torelli} for the orbifold case) can then be summarized as follows. 
First of all, the Hausdorff reduction $H\colon\mathfrak{M}_{\Lambda}\to\overline{\mathfrak{M}}_{\Lambda}$ exists, so that we have a factorization
\[
\xymatrix{
\mathfrak{M}_{\Lambda}\ar[rr]^{P_\Lambda}\ar[rd]_{H}&&\Omega_\Lambda\\
& \overline{\mathfrak{M}}_{\Lambda}\ar[ur]_{\overline{P}_\Lambda} &
},
\]
where all maps are holomorphic.
Moreover, nonseparated points give rise to bimeromoprhic primitive symplectic varieties.
Let us now fix a connected component $\mathfrak{M}^0\subset\mathfrak{M}_{\Lambda}$ and assume that $b\geq5$.
Then the stabilizer subgroup $\Stab_{O(\Lambda)}(\mathfrak{M}^0)\subset O(\Lambda)$ of $\mathfrak{M}^0$ has finite index.\footnote{If we define the {\sf monodromy group} $\mathrm{Mon}(M)$ of a primitive symplectic variety $M$ as the subgroup of $O(H^2(M,\Z)_\mathrm{tf})$ generated by all monodromies induced by locally trivial deformations of $M$, then we have $\Stab_{O(\Lambda)}(\mathfrak{M}^0)=\eta\, \mathrm{Mon}(M)\,\eta^{-1}$, for $(M,\eta)\in\mathfrak{M}^0$.}
Moreover, if one point in $\mathfrak{M}^0$ corresponds to a primitive symplectic variety with $\Q$-factorial terminal singularities\footnote{Some attention is needed in the analytic setting. For instance, the existence of a terminalization, or even the definition of $\Q$-factoriality, is not completely clear (see~\cite[Section 2.12]{BL:GlobalModuli}).}, then the same is true for all points in $\mathfrak{M}^0$ and the restriction
\[
\overline{P}_{\Lambda}|_{\overline{\mathfrak{M}^0}}\colon\overline{\mathfrak{M}^0}\longrightarrow\Omega_\Lambda
\]
is an isomorphism.\footnote{In \emph{loc.~cit.}, there is a version of the Torelli Theorem for general primitive symplectic varieties, not necessarily with $\Q$-factorial terminal singularities.}

Everything we saw in this section holds true for hyper-K\"ahler varieties. 
In fact, the definition of hyper-K\"ahler varieties can be given in the analytic setting as a primitive symplectic analytic variety, with $\Q$-factorial terminal singularities, and having regular part which is simply connected.
This last condition is preserved by locally trivial deformations: hence, if one point in a connected component $\mathfrak{M}^0$ of $\mathfrak{M}_\Lambda$ corresponds to a hyper-K\"ahler variety, then all points are hyper-K\"ahler varieties, and the Torelli Theorem holds in this context in the same form as above.
In particular, deformation-equivalence is well-defined for hyper-K\"ahler varieties, and a hyper-K\"ahler variety is deformation-equivalent to a projective one.
In fact, the projectivity criterion implies that two deformation-equivalent projective hyper-K\"ahler varieties can be connected by a projective locally trivial deformation over a connected quasi-projective base.

Properties that are more specific to hyper-K\"ahler varieties are the following.
First of all, since hyper-K\"ahler varieties are simply connected (as observed in the proof of Proposition~\ref{prop:HKisIHS}), the second cohomology group is torsion-free.
Moreover, two birational projective hyper-K\"ahler varieties are deformation-equivalent (see~\cite[Theorem 6.16]{BL:GlobalModuli}; in the smooth case this is~\cite[Theorem 4.6]{Huy:Basic}).

\subsection{Irreducible symplectic varieties and the Decomposition Theorem}\label{subsec:BBdec}

The notion of irreducible symplectic variety was introduced in~\cite{GKP:Singular}, in the context of the Beauville--Bogomolov Decomposition Theorem.
We first recall the original definition (see~\cite{BGL:Decomposition} for the analytic version).

\begin{defi}\label{def:MoreCY}
Let $X$ be a compact K\"ahler variety with canonical singularities of dimension $n\geq1$.
We say that $X$ is
\begin{enumerate}[{\rm (i)}]
    \item\label{enum:MoreCY1} {\sf irreducible Calabi--Yau} if the algebra $H^0(X',\Omega_{X'}^{[\bullet]})$ is generated  by a nowhere vanishing reflexive form in degree~n, for all quasi-\'etale covers $X'\to X$;
    \item\label{enum:MoreCY2} {\sf irreducible symplectic} if the algebra $H^0(X',\Omega_{X'}^{[\bullet]})$ is generated by a symplectic form $\sigma'\in H^0(X',\Omega_{X'}^{[2]})$, for all quasi-\'etale covers $X'\to X$.
\end{enumerate}
\end{defi}

In Definition~\ref{def:MoreCY}\ref{enum:MoreCY2}, the condition on the exterior algebra means that
\[
H^0\big(X',\Omega_X^{[p]}\big)=\begin{cases}
    \C\cdot f^*\sigma^{\wedge p/2},&\text{if $p$ is even}\\
    0,&\text{if $p$ is odd}
\end{cases}
\]
for all quasi-\'etale covers $f\colon X'\to X$.
Similarly, in Definition~\ref{def:MoreCY}\ref{enum:MoreCY1}, we have that $H^0\big(X',\Omega_X^{[p]}\big)=0$, for $0<p<n$.

Clearly, an irreducible symplectic variety in the sense of Definition~\ref{def:MoreCY}\ref{enum:MoreCY2} is irreducible symplectic in the sense of Definition~\ref{def:MoreSymplectic}\ref{enum:MoreSymplectic2}.
The converse follows immediately from the following fundamental result (see~\cite{GKKP:Differential,Dru:Decomposition,DG:Decomposition,Gue:Stability,GGK:klt,HP:Integrability,Cam:BB,BGL:Decomposition}; in the smooth case, this is~\cite[Th\'eor\`eme 1]{Bea:c10}).

\begin{theo}[Beauville--Bogomolov Decomposition Theorem]\label{thm:BB}
Any compact K\"ahler space with numerically trivial canonical class and klt singularities admits a quasi-\'etale cover which can be decomposed into a product of complex tori, irreducible Calabi--Yau varieties, and irreducible symplectic varieties, in the sense of Definition~\ref{def:MoreCY}.
\end{theo}

We could state a version of Theorem~\ref{thm:BB} completely within the projective category. 
The irreducible Calabi--Yau or irreducible symplectic factors are not unique.
By~\cite[Corollary 13.3]{GGK:klt}, when the dimension is even, they are simply connected.
As observed in Section~\ref{subsec:Def}, we can choose the irreducible symplectic factor to be strictly irreducible and, if Conjecture~\ref{conj:Rapagnetta} is true, we could choose it so that its regular locus is simply connected.
In such a case, by Proposition~\ref{prop:IHSisHK}, any $\Q$-factorial terminalization is a projective hyper-K\"ahler variety.
Hence, from the viewpoint of classification of projective varieties, Theorem~\ref{thm:BB} gives one of the main evidence to study hyper-K\"ahler varieties and their classification.
In the smooth case, irreducible symplectic manifolds are hyper-K\"ahler manifolds.

\subsection{Examples}\label{subsec:Examples}

A hyper-K\"ahler variety in dimension~$2$ is always smooth and it is nothing but a {\sf K3 surface}.\footnote{In this paper, to be coherent with the main definition, all K3 surfaces are smooth and projective.}
By a celebrated result by Kodaira~\cite{Kod:DefoK3}, all K3 surfaces are deformation-equivalent. They have $b_2=22$.

To get other examples of hyper-K\"ahler varieties, starting from a K3 surface $S$, we can look at its Hilbert scheme of points $S^{[n]}$, for $n\geq1$: these give examples of projective\footnote{Analytic as well, if we allow the K3 surface to be analytic.} hyper-K\"ahler manifolds of dimension $2n$~\cite{Bea:c10}. They are all deformation-equivalent: a hyper-K\"ahler manifold deformation-equivalent to a Hilbert scheme of points on a K3 surface will be called of K3$^{[n]}$-{\sf type}.
These examples have $b_2=23$ and Fujiki constant~1.
Similarly, in \emph{loc.~cit.}, another class of examples of hyper-K\"ahler manifolds in all dimension has been constructed by starting from an abelian surface $A$, considering its Hilbert scheme of points $A^{[n+1]}$, and then the kernel $K_n(A)$ of the summation morphism $A^{[n+1]}\to A$.
These varieties, called {\sf generalized Kummer varieties}, have again dimension $2n$, they are all deformation-equivalent, but in a different deformation class with respect to those of K3$^{{[n]}}$-type. We call them of Kum$_{n}$-{\sf type}.
They have $b_2=7$ and Fujiki constant~$n+1$.

These two constructions are difficult to generalize directly, unless we reinterpret them as moduli spaces of stable sheaves (on lower dimensional K-trivial varieties), or look at their Lagrangian fibration structure, or as group quotients.
In fact, these generalizations have been the main strategy to construct examples of hyper-K\"ahler varieties in the past years.
In this section we will analyze in detail moduli spaces on K3 surfaces; the first steps towards a general theory in higher dimension are in Section~\ref{sec:atomic}.
Compactifications of Lagrangian fibrations will be treated in detail in Section~\ref{sec:Lagrangian}.
Group quotients have been studied in detail in lower dimension, the idea being to find a symplectic group action on a hyper-K\"ahler variety and then take its terminalization. 
This approach has been first developed in~\cite{Fuj:Primitive}, and then extended in works by many people (see, e.g.,~\cite{FM:Betti,BGMM:Terminalization} for the latest results in dimension~$4$).
We will not touch on these constructions in this paper, and we refer to~\cite{Per:Survey} for a more detailed survey.
A 20-dimensional (singular) new example has been recently announced in~\cite{LY:Quotient}; it would be interesting to compare this with the constructions in Section~\ref{sec:Lagrangian}.

\subsubsection*{The Mukai structure}
Let $S$ be a smooth projective K3 surface over $\C$.
We start by recalling the Mukai Hodge structure on the full singular cohomology of $S$
\begin{equation}\label{eq:CohomK3}
H^*(S,\Z)=H^0(S,\Z)\oplus H^2(S,\Z)\oplus H^4(S,\Z).    
\end{equation}
This is a torsion-free abelian group of rank~24.
We often denote an element $\bv=(v_0,v_1,v_2)\in H^*(S,\Z)$ with respect to the decomposition~\eqref{eq:CohomK3}. 
We define the {\sf Mukai pairing} on $H^*(S,\Z)$ by
\[
(\bv,\bv')\coloneqq\int_S \left(v_1\cup v_1' - v_0v_2'-v_2v_0'\right)\in\Z,\qquad \bv,\bv'\in H^*(S,\Z).
\]
We denote
\[
\widetilde{\Lambda}_\mathrm{K3}\coloneqq \left(H^*(S,\Z),(-,-)\right)
\]
the corresponding lattice; it is isometric to $U^{\oplus 4}\oplus E_8(-1)^{\oplus 2}$ as a lattice, but we will not need this fact in this paper except that it is \emph{even}.
Finally, we endow this with the weight~2 Hodge structure coming from the one of $S$:
\[
H^{2,0}(S)\longhookrightarrow H^2(S,\C)\longhookrightarrow H^*(S,\C). 
\]
The {\sf Mukai structure} is the triple
\[
\widetilde{H}(S,\Z)\coloneqq \left(H^*(S,\Z),(-,-),H^{2,0}(S)\right).
\]
We notice that, with respect to this Hodge structure, we have
\[
\widetilde{H}_\mathrm{alg}(S,\Z)\coloneqq\widetilde{H}^{1,1}(S)=H^0(S,\Z)\oplus \NS(S)\oplus H^4(S,\Z),
\]
where $\NS(S)$ denotes the N\'eron--Severi group of $S$.

Let us denote by $\Db(S)\coloneqq\Db(\coh(S))$ the bounded derived category of coherent sheaves on $S$. For an object $E\in\Db(S)$, we can define its {\sf Mukai vector}
\[
v(E)\coloneqq\ch(E).\sqrt{\td_S}\in \widetilde{H}_\mathrm{alg}(S,\Z).
\]
This extends to a group homomorphism $v\colon K(S)\to \widetilde{H}_\mathrm{alg}(S,\Z)$, form the Grothendieck group of $S$.
The Hirzebruch--Riemann--Roch Theorem then becomes
\begin{equation}\label{eq:HRR}
\chi_S(E,E')\coloneqq\sum_j (-1)^j \dim_\C \Hom_{\Db(S)}(E,E'[j])=-(v(E),v(E')),
\end{equation}
for all $E,E'\in\Db(S)$.

\begin{rema}
The Mukai structure naturally appears in the so called \emph{Derived Torelli Theorem} which detects when two K3 surfaces have equivalent derived categories.
We will not need this, but it is important in the development of the theory.
We refer to~\cite{Muk:Moduli,Orl:K3,HMS:Orientation} for more details.
\end{rema}

\subsubsection*{Bridgeland stability conditions and moduli spaces}
The setting to state the result about moduli spaces of stable sheaves on K3 surfaces is actually Bridgeland stability~\cite{Bri:Stab} on $\Db(S)$.
We will not need the full definition.
For us a {\sf Bridgeland stability condition} on $\Db(S)$ is a pair $\sigma=(Z,\cP)$, where
\[
Z\colon \widetilde{H}_\mathrm{alg}(S,\Z) \longrightarrow \C
\]
is a group homomorphism (called the {\sf central charge}) and
\[
\cP=\bigcup_{\phi\in\R}\cP(\phi)\subset\Db(S)
\]
is a full subcategory with the compatibility
\[
Z(v(E))\in\R_{>0}\,e^{i\pi\phi}, \quad\text{ for all }\phi\in\R.
\]
Objects in $\cP(\phi)$ are called $\sigma$-{\sf semistable} of phase $\phi$.
Bridgeland stability conditions are required to satisfy several properties, including the existence of Harder--Narasimhan filtrations in semistable objects, the support property (which in particular implies that $\sigma$-{\sf stable} objects are well defined and behave as expected)\footnote{The category $\cP(\phi)$ is an abelian category, once we add the trivial object. The support property implies that it is Artinian and Noetherian; stable objects are the simple objects in this category. See~\citestacks{0FCD} for the basic definitions regarding abelian categories and the Jordan--H\"older property. The support property was introduced in~\cite{KS:Stability}, complementing the original definition of locally finite stability condition in~\cite[Definition 5.7]{Bri:Stab} (see also~\cite[Proposition B.4]{BM:LocalP2}).}, and boundedness and openness for an object to be in $\cP(\phi)$.
We do not need the explicit form of these conditions, for which we refer to~\cite{Bri:Stab,BLMNPS:Family,MS:Lectures}.
The main consequence of the definition is that the set of stability conditions $\Stab(S)$ has naturally a topology for which the forgertful morphism $(Z,\cP)\mapsto Z$ is a local homeomorphism. In particular, $\Stab(S)$ has the structure of complex manifold of dimension $\rk(\widetilde{H}_\mathrm{alg}(S,\Z))=\rho(S)+2$ and we can explicitly describe the image of the forgetful morphism (see~\cite{Bri:Stab,BB:Pic1,Bay:ShortProof}).

Another key point is that, by fixing $\phi\in\R$ and a vector $\bv\in\widetilde{H}_\mathrm{alg}(S,\Z)$, the set of (S-equivalence classes of) $\sigma$-semistable objects in $\cP(\phi)$ with Mukai vector $\bv$ has a natural structure of proper algebraic space over $\C$ (see~\cite[Theorem 1.3]{BM:Projectivity} and~\cite{Alper:GoodModuli,AHLH:Moduli}). 
We will denote this by $M_{S,\sigma}(\bv)$, by dropping the reference to the phase.
The open subscheme parameterizing $\sigma$-stable objects will be denoted by $M_{S,\sigma}^\mathrm{st}(\bv)$.

We denote by $\Stab^\dagger(S)$ the connected component of the space of Bridgeland stability conditions on $\Db(S)$ described in~\cite{Bri:K3}.
The main result is the following.
It comes out of the work of many people, starting from Beauville, Mukai, Kuleshov, O'Grady, Huybrechts, Yoshioka, and Perego--Rapagnetta in the case of stable sheaves~\cite{Bea:c10,Muk:Moduli,OG:Moduli,Huy:Basic,Yos:Abelian,OG10,PR:K3}.
The last version on stability conditions has been established in~\cite{Toda:K3,BM:Projectivity,BM:MMP,AS:Update}.
A direct proof, by using Bridgeland stability conditions from the beginning, is in~\cite{Bot:Yoshioka}.

Let us fix a vector $\bv=m\bv_0\in\widetilde{H}_\mathrm{alg}(S,\Z)$, for $m\geq1$, where $\bv_0$ is primitive.
The first part of the result is a nonemptiness statement.
First of all, we observe that an elementary application of~\eqref{eq:HRR} and Serre duality shows that if $E\in M_{S,\sigma}^\mathrm{st}(\bv)$, then $\bv^2\geq-2$:
\begin{equation}\label{eq:SerreDualityChi}
\bv^2 = -\chi_S(E,E) = -\underbrace{\hom_S(E,E)}_{=1}+\underbrace{\ext^1_S(E,E)}_{\geq0}-\underbrace{\ext^2_S(E,E)}_{=\hom_S(E,E)} \geq -2.    
\end{equation}

The key point is that a converse holds (see, e.g.,~\cite[Theorem 2.15]{BM:MMP}):

\begin{theo}[Mukai--Kuleshov--O'Grady--Yoshioka]\label{thm:NonEmpty}
If $\bv_0^2\geq-2$, then $M_{S,\sigma}(\bv)\neq\emptyset$, for all $\sigma\in\Stab^\dagger(S)$. 
\end{theo}

We can be more precise about stable objects as well.
Let us assume from now on that $\bv_0^2\geq2$ (the cases $\bv_0^2=-2,0$ are slightly special: see Remark~\ref{rmk:-20}).
Let $\sigma\in\Stab^\dagger(S)$ be a $\bv$-\emph{generic} stability condition, namely with the property that $M_{S,\sigma}(\bv_0)=M_{S,\sigma}^\mathrm{st}(\bv_0)$ and the only properly semistable objects in $M_{S,\sigma}(\bv_0)$ are S-equivalent to direct sums of objects in $M_{S,\sigma}(\bv_0)$.
Under our assumption, there is an open subset of $\Stab(S)$ of stability conditions which are $\bv$-generic.
Then a precise form of Theorem~\ref{thm:NonEmpty} is: $M_{S,\sigma}^\mathrm{st}(\bv)\neq\emptyset$, for $\sigma\in\Stab^\dagger(S)$ a $\bv$-generic stability condition.

The second part of the result is that these moduli spaces give examples of projective hyper-K\"ahler varieties in any dimension.\footnote{They also come with a natural polarization $\ell_\sigma(\bv)$ (see~\cite[Theorem 1.3]{BM:Projectivity}).}
There are three cases to consider: the case in which $m=1$, O'Grady's celebrated example, and all the rest.

We first deal with the case $m=1$.
This is~\cite[Thoerem 8.1]{Yos:Abelian} (see~\cite[Theorem 6.10 \& Section 7]{BM:Projectivity}, for the statement below in the context of Bridgeland stability conditions).

\begin{theo}[Beauville--Mukai--Huybrechts--O'Grady--Yoshioka]\label{thm:Yoshioka}
Let us assume $\bv=\bv_0$, namely $m=1$, and $\bv^2\geq2$.
Let $\sigma\in\Stab^\dagger(S)$ be a $\bv$-generic stability condition.
Then $M\coloneqq M_{S,\sigma}(\bv_0)=M_{S,\sigma}^\mathrm{st}(\bv_0)$ is a smooth projective hyper-K\"ahler manifold of K3$^{[n]}$-type, where $n\coloneqq(\bv^2+2)/2$.
Moreover, we have an Hodge isometry
\[
H^2(M,\Z)\xlongrightarrow{\cong}\bv^{\perp}\subset\widetilde{H}(S,\Z),
\]
where on the left hand side we are considering $H^2(M,\Z)$ with its natural weight $2$ Hodge structure and the Beauville--Bogomolov--Fujiki form, and on the right hand side $\bv^{\perp}$ with the induced Mukai structure from $\widetilde{H}(S,\Z)$.
\end{theo}

The smoothness of the moduli space $M=M_{S,\sigma}(\bv)$ at a point $E\in\Db(S)$ is controlled by its tangent space, $\Ext^1(E,E)$, and by the obstruction morphism into $\Ext^2(E,E)$.
In the K3 surface case, this morphism is quadratic and coincides with the Yoneda composition
\[
\Ext^1(E,E)\longrightarrow\Ext^2(E,E)\qquad f\longmapsto f[1]\circ f.
\]
Hence, if $E$ is $\sigma$-stable, then $\Ext^2(E,E)$ is $1$-dimensional, as remarked already in~\eqref{eq:SerreDualityChi}, and it is a general fact that the Yoneda pairing is skew-symmetric, when composed with the trace morphism
\[
\Tr\colon\Ext^2(E,E)\xlongrightarrow{\cong} H^2(S,\cO_S)\cong\C
\]
(more details about the trace morphism will be in Section~\ref{subsec:Hyperholomorphic}).
Hence, the fact that $M$ is smooth and symplectic in Theorem~\ref{thm:Yoshioka} follows from general deformation theory and Serre duality (see~\cite{Muk:Symplectic}).
In general, the same approach allows us to study the singularities of the moduli space as well, in Theorem~\ref{thm:OG10} and Theorem~\ref{thm:PR} below. In higher dimensions, this is more complicated: we will see how to deal with this in Section~\ref{sec:atomic}.

The next case is O'Grady's exceptional example.
This is~\cite{OG10} (the version for stability conditions is~\cite[Theorem 1.1]{LPZ:Elliptic}\footnote{To be precise in \emph{loc.~cit.~}the authors prove the result for the Kuznetsov component associated to a cubic fourfold. Their proof in particular adapts to Bridgeland stability conditions on K3 surfaces.}).

\begin{theo}[O'Grady]\label{thm:OG10}
Let us assume that $\bv=2\bv_0$ and $\bv_0^2=2$.
Then, for a $\bv$-generic stability condition $\sigma\in\Stab^\dagger(S)$, the moduli space $M_{S,\sigma}(\bv)$ admits a resolution of singularities $\widetilde{M}$ which is a smooth projective hyper-K\"ahler manifold of dimension~10. 
\end{theo}

The second cohomology group of $\widetilde{M}$ has been described in~\cite{Rap:OG10}; in particular, $b_2(\widetilde{M})=24$ and $c_{\widetilde{M}}=1$.
There is also a similar description of this, as in Theorem~\ref{thm:Yoshioka} and Theorem~\ref{thm:PR}, in terms of the extended Mukai lattice (see Section~\ref{sec:AtomicObjects} and Theorem~\ref{thm:Bottini}).
Hyper-K\"ahler manifolds arising from Theorem~\ref{thm:OG10} are all deformation equivalent; we call them of OG10-{\sf type}.

Finally, in all other cases, we get examples of projective hyper-K\"ahler varieties, which are not smooth.
This is~\cite[Theorem 1.10]{PR:K3} (based on~\cite{KLS:K3}; for the statement about stability conditions, see~\cite[Sections 5 and 6]{AS:Update}).

\begin{theo}[Perego--Rapagnetta]\label{thm:PR}
Let us assume that $\bv=m\bv_0$, with $m\geq2$ and $\bv_0^2\geq2$, excluding the case in which $m=2$ and $\bv_0^2=2$.
Then, for a $\bv$-generic stability condition $\sigma\in\Stab^\dagger(S)$, the moduli space $M_{S,\sigma}(\bv)$ is a projective hyper-K\"ahler variety of dimension $2n\coloneqq\bv^2+2$.
Moreover, we have an Hodge isometry
\[
H^2(M,\Z)\xlongrightarrow{\sim}\bv^{\perp}\subset\widetilde{H}(S,\Z),
\]
where on the left hand side we are considering $H^2(M,\Z)$ with its natural weight 2 Hodge structure and the Beauville--Bogomolov--Fujiki form, and on the right hand side $\bv^{\perp}$ with the induced Mukai structure from $\widetilde{H}(S,\Z)$.
\end{theo}

When fixing the pair $(m,n)$, by~\cite[Theorem 1.7]{PR:K3}, all such moduli spaces are (locally trivially) deformation-equivalent; we call them of PR$_{m,n}$-{\sf type}.
They have $b_2=23$ and Fujiki constant~1.

\begin{rema}\label{rmk:-20}
The case $\bv_0^2=0$ is slightly different.
Indeed, Theorem~\ref{thm:Yoshioka} still holds: the moduli space $M$ is a K3 surface in this case, but $H^2(M,\Z)$ is now Hodge isometric to $\bv^\perp/\bv$.
If $m\geq2$, then the moduli space $M_{S,\sigma}(\bv)$ is isomorphic to $\Sym^m(M_{S,\sigma}(\bv_0))$ (see, e.g., \cite[Lemma 7.2]{BM:Projectivity}).
If $\bv_0^2=-2$, the moduli space (in both Theorem~\ref{thm:Yoshioka} and Theorem~\ref{thm:PR}) is just a point.
\end{rema}

\begin{rema}\label{rmk:Gieseker}
As observed in~\cite[Proposition 14.1]{Bri:K3} and \cite[Section 6]{Toda:K3}, if we fix an ample divisor $H$ on $S$, for a \emph{positive} vector $\bv$\footnote{A vector $\bv=(r,D,s)\in\widetilde{H}_\mathrm{alg}(S,\Z)$ is positive if, either $r>0$, or $r=0$ and $D$ is an effective divisor, or $r=D=0$ and $s>0$.}, we can choose a Bridgeland stability condition for which semistable objects in $\Db(S)$ with Mukai vector $\bv$ coincide with \emph{$H$-Gieseker semistable} sheaves of Mukai vector $\bv$.
When the ample divisor $H$ is $\bv$-generic, the stability condition $\sigma$ is also $\bv$-generic.
Hence, the Bridgeland stability statements include the case of $H$-Gieseker semistable sheaves as a particular case.
We will denote moduli spaces of $H$-Gieseker semistable sheaves by $M_{S,H}(\bv)$.
\end{rema}

\begin{rema}\label{rmk:twisted}
There are analogous results for the twisted derived category $\Db(S,\alpha)$, where $S$ is a projective K3 surface and $\alpha\in\Br(S)$ is a Brauer class (see~\cite{HS:Twisted,Yos:Twisted}).
We will see a further generalization in Section~\ref{sec:categories} where actual K3 surfaces (or their twisted versions) will be replaced by more abstract triangulated categories with special homological properties making them resemble the derived category of a K3 surface (or of a family of K3 surfaces).
\end{rema}

%%%%%%%%%%%%%%%%%%

\section{Lagrangian fibrations}\label{sec:Lagrangian}

The goal of this section is to review methods for the construction and classification of hyper-K\"ahler varieties, by using their structure of Lagrangian fibration.
More precisely, we will review the fundamental conjecture on the existence of Lagrangian fibrations (in Section~\ref{subsec:SYZ}), see how this would have quite strong consequences on the classification of topological types of hyper-K\"ahler varieties (in Section~\ref{subsec:EFGMS}), and finally state a result which allows us to produce interesting examples of hyper-K\"ahler varieties by compactifying a quasi-projective variety with a Lagrangian fibration (in Section~\ref{subsec:Sacca}).
Examples are in Section~\ref{subsec:ExamplesLagrangian}.

We refer to~\cite{Saw:Survey,HM:Lagrangian,MSY:Abelian} for more exhaustive references on the subject.

\subsection{Definition and the main conjecture}\label{subsec:SYZ}

%In this section we work again with quasi-projective varieties.
We start by reviewing the notion of Lagrangian fibration in the algebraic setting.
Recall that a morphism $f\colon X\to Y$ between integral, normal, quasi-projective varieties is called a \emph{fibration} if it is surjective and with connected fibers.
The general fiber $F$ is then normal and integral, with $F_\mathrm{sing}=F\cap X_\mathrm{sing}$.
Moreover, when $X$ has canonical (resp., terminal) singularities, $F$ has canonical (resp., terminal) singularities as well (see, e.g.,~\cite[Lemma 28]{Sch:Fujiki}).

\begin{defi}\label{def:Lagrangian}
Let $(M,\sigma)$ be a quasi-projective symplectic variety.
\begin{enumerate}[{\rm (i)}]
    \item An integral closed subvariety $Z\subset M$ is called {\sf Lagrangian} if
    \begin{itemize}
        \item $Z\not\subset M_\mathrm{sing}$,
        \item $\dim(Z)=\frac 12 \dim(M)$,
        \item $\sigma|_{Z_\mathrm{reg}\cap M_\mathrm{reg}}\equiv0$.
    \end{itemize} 
    \item A fibration $f\colon M\to B$ is called a {\sf Lagrangian fibration} if it is proper and the general fiber is a Lagrangian subvariety.
\end{enumerate}
\end{defi}

The surprising result is that for primitive symplectic varieties any fibration is a Lagrangian fibration and the general fiber is actually an abelian variety.
This is~\cite{Mat:Base,Mat:Addendum,Mat:Equidimensionality,Hwa:Pn,ShenYin} in the smooth case, extended to the singular case in~\cite{Mat:Base2},~\cite[Theorems 3 \& 4]{Sch:Fujiki}, and~\cite[Theorem 1.1 \& Remark 1.2]{FSY:Intersection}.
The statement in the general analytic setting is~\cite[Theorem 2.10]{KL:NonHyperb}.

\begin{theo}[Matsushita, Hwang, Schwald]\label{thm:Matsushita}
Let $M$ be a projective primitive symplectic variety of dimension $2n$.
Let $f\colon M\to B$ be a fibration onto a normal projective variety $B$ with $0<\dim(B)<2n$.
Then:
\begin{enumerate}[{\rm (a)}]
    \item $B$ is a $\Q$-factorial klt variety of dimension~$n$ with Picard number $\rho(B)=1$.
    \item $f(M_\mathrm{sing})\subset B$ is a proper closed subset; in particular, the general fiber of $f$ is smooth and entirely contained in $M_\mathrm{reg}$.
    \item $f$ is a Lagrangian fibration and the general fiber is an abelian variety of dimension $n$.
    \item Every irreducible component of any fiber of $f$ (with the reduced induced scheme structure) is a Lagrangian subvariety of $M$.
\end{enumerate}
If we further assume that $M$ is irreducible symplectic, then
\begin{enumerate}[{\rm (a)}, resume]
    \item $B$ is Fano (with $H^*(B,\Q)\cong H^*(\P^n,\Q)$ when $b_2(M)\geq5$\footnote{If $M$ is smooth, then the assumption on $b_2(M)$ is not necessary, by~\cite[Theorem 0.4]{ShenYin}.}) and
    \item if $B$ is smooth, then $B\cong\P^n$.
\end{enumerate}
\end{theo}

The main open problem in the theory of symplectic varieties is the so called \emph{SYZ Conjecture}: Every primitive symplectic variety with $b_2\geq5$ is deformation-equivalent to a primitive symplectic variety with a Lagrangian fibration.
The precise version of this conjecture is the following, which is also a special case of the \emph{Generalized Abundance Conjecture} (see~\cite{Kaw:Abundance}).

\begin{conj}[SYZ Conjecture]\label{conj:SYZ}
Let $M$ be a projective primitive symplectic variety and let $L$ be a nef line bundle on $M$. Then $L$ is semiample.    
\end{conj}

If $q_M([L])>0$, then Conjecture~\ref{conj:SYZ} holds, by the Kawamata--Viehweg Vanishing Theorem.
By Theorem~\ref{thm:Matsushita}, if we take a nonzero nef line bundle with $q_M([L])=0$, Conjecture~\ref{conj:SYZ} gives a Lagrangian fibration $f\colon M\to B$ associated to (multiples of) $L$.
There are more precise versions of the above conjecture in the hyper-K\"ahler case (for instance, in the smooth case, there is the expectation that $B$ is always smooth as well, there are no multiple fibers in codimension~1, etc.).
We simply remark here that Conjecture~\ref{conj:SYZ} is true (in the strongest possible form) for all known examples of deformation classes of smooth projective hyper-K\"ahler manifolds (see~\cite{BM:MMP,Markman:Lagrangians,Yos:MMP,MO:OG10,MR:OG6}).

Conjecture~\ref{conj:SYZ} has strong implications.
We will discuss briefly some of these in Section~\ref{subsec:EFGMS}.
An idea of Markman is also that we can conversely construct examples of symplectic varieties by compactifying special Lagrangian fibrations.
This is the content of the next section.

\subsection{The Sacc\`a Theorem}\label{subsec:Sacca}

It was first observed in~\cite{LSV:OG10} that once we have a symplectic variety with a morphism onto a `large' open subset of a projective base, then we can compactify this to another symplectic variety.
We will see this example, the Laza--Sacc\`a--Voisin system, in Section~\ref{subsec:ExamplesLagrangian}.
The general result is quite strong in fact and it is proved in~\cite{Sac:Lagrangian}.

\begin{theo}[Sacc\`a]\label{thm:Sacca}
Let $B$ be an integral normal projective variety, and let $\emptyset\neq U\subset V\subset B$ be nonempty open subsets.    
Let $(M_V,\sigma_V)$ be an integral quasi-projective symplectic variety with $\Q$-factorial, terminal singularities, and let $f_V\colon M_V\to V$ be a fibration.
Let us assume:
\begin{enumerate}[{\rm (i)}]
    \item\label{enum:Sacca1} $\codim_B(V)\geq2$;
    \item\label{enum:Sacca2} $\sigma_V$ extends to a regular 2-form on a regular compactification of $M_{V,\mathrm{reg}}$;
    \item\label{enum:Sacca3} if $M_U\coloneqq f_V^{-1}(U)$, then the restriction $f_U\colon M_U\to U$ is a Lagrangian fibration.
\end{enumerate}
Then there exists an integral projective symplectic variety $(M,\sigma)$ with $\Q$-factorial, terminal singularities, and a Lagrangian fibration $f\colon M\to B$ such that:
\begin{itemize}
    \item $M_V\subset M$ is an open subset and $\sigma$ is an extension of $\sigma_V$;
    \item $f$ is an extension of $f_V$.
\end{itemize}
\end{theo}

Some remarks are in order about Theorem~\ref{thm:Sacca}.
First of all, we stress the fact that $f_V$ is a fibration (hence surjective, with connected fibers), but it is not assumed to be proper.
The assumptions only ask for a symplectic form on $M_V$ and a Lagrangian fibration structure (in particular, proper) over a possibly smaller open subset $U\subset V$.
The important assumption is that $V$ is `large' in $B$, namely it has complement of codimension at least~2.
The fact that $\sigma_V$ extends to a regular 2-form on a (hence all) regular compactification seems difficult to check, but in the examples is often automatic: the 2-forms are generally constructed by using cycles, hence they extend to any compactification.

Theorem~\ref{thm:Sacca} can be improved, to obtain irreducible symplectic varieties, or even hyper-K\"ahler varieties as compactifications.
Several extensions were proved in~\cite{LLX:Lagrangian}.
In particular, we mention~\cite[Theorem~3.6]{LLX:Lagrangian}: by using the Beauville--Bogomolov Decomposition Theorem, they show that if we further assume, in Theorem~\ref{thm:Sacca}, that the very general fiber of $f_U$ is a simple abelian variety and that $f_U$ is not isotrivial, then the compactification $M$ is an irreducible symplectic variety.

\subsection{Examples}\label{subsec:ExamplesLagrangian}

There are two key examples of Lagrangian fibrations on projective hyper-K\"ahler manifolds: the Beauville--Mukai system and the Laza--Sacc\`a--Voisin system.
We discuss them briefly in this section, and then present as well new singular examples based directly on Sacc\`a's Theorem.

\subsubsection*{The Laza--Sacc\`a--Voisin 2-form}
In all examples, the 2-form is constructed by a general procedure, explained in~\cite[Section 1.1]{LSV:OG10}.
These forms have the property that they always extend and they are closed (see~\cite[Theorem 1.2(iii),(iv)]{LSV:OG10}), namely assumption~\ref{enum:Sacca2} in Theorem~\ref{thm:Sacca} is always satisfied.
We just give an idea of their construction, since it is quite intuitive.
The general setting is the following.
We fix an integer $k\geq1$ and we let $(X,H)$ be a smooth projective variety of dimension~$2k$, with a very ample divisor $H$.
We consider the open subset $U\subset |H|$ parameterizing smooth hyperplane sections and $g\colon\cY_U\to U$ the associated smooth projective fibration.
We assume the following: for all $[Y]\in U$, $h^{p,q}(Y)=0$, when $p+q=2k-1$ and $(p,q)\neq(k,k-1),(k-1,k)$.
We can then consider the smooth projective fibration $f_U\colon \cJ_U\to U$ whose fibers are the intermediate Jacobians $J^{2k-1}(Y)\coloneqq H^{k-1}(Y,\Omega_Y^k)^\vee/H_{2k-1}(Y,\Z)$.
Then, for any $\eta\in H^{k+1,k-1}(X)$ with $\eta |_Y\equiv 0$, we get a 2-form $\sigma_\eta$ on $\cJ_U$ whose restriction on the fibers of $f_U$ is trivial.

The idea is the following.
For simplicity we assume that there exists a codimension-$k$ cycle $\cZ\in \mathrm{CH}^k(\cJ_U\times_U\cY_U)_\Q$ such that the projection $\alpha_0$ of its class $[\cZ]\in H^{2k}(\cJ_U\times_U\cY_U)$ to $H^0(U,R^{2k}(f,g)_*\Q)$ has the property that
\[
\alpha_0^*\colon R^{2k-1}g_*\Q \xlongrightarrow{\sim}R^1f_*\Q
\]
is an isomorphism.
Then, since the morphism $q\colon \cJ_U\times_U\cY_U \to\cJ_U\times X$ is proper, we can define $\cZ_q\coloneqq q_*\cZ\in\mathrm{CH}^{k+1}(\cJ_U\times X)_\Q$, take its Dolbeault cohomology class $[\cZ_q]\in H^{k+1}(\cJ_U\times X, \Omega^{k+1}_{\cJ_U\times X})$, and finally define
\[
\sigma_\eta\coloneqq [\cZ_q]^*\eta\in H^0(\cJ_U,\Omega^2_{\cJ_U})
\]
as our 2-form.

The first issue is to show that this 2-form is in fact symplectic.
This becomes a cohomology computation (see~\cite[Theorem 1.2(ii)]{LSV:OG10}).
More precisely, for all $[Y]\in U$, by using the assumptions, the class $\eta$ induces a class $\widetilde{\eta}\in H^{k-1}(Y,\Omega_Y^k(-H_Y))$, where $H_Y$ denotes the restriction of $H$ to $Y$.
Then the 2-form $\sigma_\eta$ is nondegenerate if the morphism given by multiplication by $\widetilde{\eta}$
\begin{equation}\label{eq:SymplecticCheck}
H^0(Y,\cO_Y(H_Y)) \longrightarrow H^{k-1}(Y,\Omega_Y^k)    
\end{equation}
is an isomorphism.
The second issue, which is more important, is that there is a larger open subset $V\supset U$ where the extension of $\sigma_\eta$ is also symplectic (in general, we must allow some singularities of $Y$; for instance, nodal singularities).

\subsubsection*{The Beauville--Mukai system}
Let $(S,H)$ be a polarized K3 surface of genus $g=1+H^2/2\geq2$.
Let us assume, for simplicity, that the Picard group of $S$ is of rank~1, $\Pic(S)=\Z\cdot H$.
The Beauville--Mukai system can be constructed in two ways, either by compactification or by moduli spaces of torsion sheaves on $S$.

The easiest way is to apply the previous construction: here $X=S$ satisfies all the assumptions (with $k=1$). Hence, we get a 2-form $\sigma_U$ on the relative Jacobian fibration $f_U\colon\cJ_U\to U$, which in this case it is not too difficult to check it is actually symplectic.
To apply Theorem~\ref{thm:Sacca}, the idea is to consider the Jacobian fibration over the locus $V$ of curves $C\in |H|$ with nodal singularities.
The open subset $V$ has indeed codimension~2 in $|H|$, and the fact that the extension $\sigma_V$ is symplectic comes as in~\cite[Section 1.4]{LSV:OG10}, by a similar isomorphism as in~\eqref{eq:SymplecticCheck} with logarithmic forms on the resolution of $Y=C$.
In the end, the symplectic compactification $\cJ_S\to|H|\cong\P^g$ which comes out of Theorem~\ref{thm:Sacca} is actually a smooth projective hyper-K\"ahler manifold of K3$^{[g]}$-type.

Another way to see this is to use moduli spaces of sheaves on the K3 surface $S$, as in Section~\ref{subsec:Examples}.
In fact, we can consider the Mukai vector $\bv=(0,H,1-g)$.
Then the moduli space of torsion sheaves of rank~1 supported on the curves $C\in |H|$ is a projective hyper-K\"ahler manifold of K3$^{[n]}$-type, by Theorem~\ref{thm:Yoshioka} (by taking a stability condition which gives Gieseker stability on sheaves).
It is not difficult to see that this also gives a compactification of $\cJ_U$, and they are actually isomorphic (given our assumption that the K3 surface has Picard rank~1).

\subsubsection*{The Laza--Sacc\`a--Voisin system}
Let $W\subset\P^5$ be a smooth projective cubic fourfold\footnote{In~\cite{LSV:OG10} the cubic is assumed to be general; see~\cite{Sac:IntermediateJac} for the extension to all cubics.}, and let us choose $H\coloneqq\cO_W(1)$ the hyperplane section.
We apply the Laza--Sacc\`a--Voisin construction to $(W,H)$ (here $k=2$), and obtain a $2$-form $\sigma_U$ on the intermediate Jacobian fibration $f_U\coloneqq\cJ_U\to U$. 
Again, it is not so difficult to prove that $\sigma_U$ is symplectic.
Similarly to the Beauville--Mukai system, we can extend this fibration to a larger $V$ by including nodal cubic hypersurfaces in $W$ and check that the extended 2-form $\sigma_V$ is again symplectic.
The symplectic compactification $\cJ_W$ which arises from Theorem~\ref{thm:Sacca} is a smooth projective hyper-K\"ahler manifold of OG10-type.
This proves a conjecture by Donagi--Markman and Markushevich.

It is more difficult in this situation to reinterpret this construction in terms of moduli spaces, but it is true. 
This is Theorem~\ref{thm:Bottini} below, and it will be the subject of Section~\ref{sec:atomic}.
A posteriori, this is similar to the Beauville--Mukai system: it is isomorphic to a moduli space of atomic sheaves on the variety of lines on the cubic fourfold (which is a projective hyper-K\"ahler manifold of K3$^{[2]}$-type, by~\cite{BD:Cubic}), when $W$ is general.

\subsubsection*{Further examples}
A slightly different series of examples still comes from geometry.
In \cite{LLX:Lagrangian}, the authors still fix a cubic fourfold $W$ as in the previous section.
The system of intermediate Jacobians is now obtained by looking at cubic fivefolds $Y$ containing $W$, but the construction is more complicated (based on~\cite{IM:Cubic,Marku:Integrable}).
They finally obtain an example of \emph{singular} projective irreducible symplectic variety $M_{\mathrm{LLX}}$, with $\Q$-factorial terminal singularities, $b_2(M_{\mathrm{LLX}})\geq24$, dimension~42, and with a Lagrangian fibration
\[
f\colon M_{\mathrm{LLX}}\longrightarrow \P(1^{15},2^6,3)
\]
over a weighted projective space.
We do expect that $b_2(M_{\mathrm{LLX}})=24$ and that it is a projective hyper-K\"ahler variety.

Similarly, in~\cite{MOS:GM} (based on~\cite{IM:GM}), we consider a Gushel--Mukai sixfold $W$ and the very ample divisor $H=\cO_W(1)$.
The construction of Laza--Sacc\`a--Voisin applies again, but the 2-form is now degenerate. This can be fixed by quotienting by the natural $\mu_2$-action coming from $X$. Finally, we obtain another example of \emph{singular} projective irreducible symplectic variety $M_{\mathrm{MOS}}$, with $\Q$-factorial terminal singularities, $b_2(M_{\mathrm{MOS}})\geq24$, dimension~20, and with a Lagrangian fibration
\[
f\colon M_{\mathrm{MOS}}\longrightarrow \P(1^{10},2)
\]
over a weighted projective space.
Again, we expect that $b_2(M_{\mathrm{MOS}})=24$ and that it is a projective hyper-K\"ahler variety.

These examples should be thought as singular analogues of hyper-K\"ahler manifolds of OG10-type.
There is yet no interpretation of these as moduli spaces, although work on this direction is in~\cite{GL:Immersed}, again by using atomic sheaves.
We will come back to this in Sections~\ref{subsec:ExamplesAtomic} and~\ref{subsec:ModuliSpaces}.

\subsection{Recent major advances}\label{subsec:EFGMS}

We finish this section on Lagrangian fibrations, by reviewing two very recent results from~\cite{EFGMS:Boundedness}.
The main result is that, if we assume that Conjecture~\ref{conj:SYZ} holds, namely that Lagrangian fibrations exist, then many other conjectures in the theory of hyper-K\"ahler varieties hold as well.

More precisely, let us assume that the SYZ Conjecture holds for all primitive symplectic varieties. Then, as proved in~\cite[Corollary 1.3]{EFGMS:Boundedness}, the number of (locally trivial) deformation classes of primitive symplectic varieties, of fixed dimension and $b_2\geq5$, is finite.
This is the first general result towards a possible classification of projective hyper-K\"ahler varieties.
Another striking consequence of the SYZ Conjecture is the proof of Conjecture~\ref{conj:Rapagnetta}~\cite[Theorem F]{EFGMS:Boundedness}.
More precisely, if $M$ is a projective irreducible symplectic variety admitting a Lagrangian fibration, then the regular locus $M_\mathrm{reg}$ has finite fundamental group.

We conclude by mentioning that, when the Lagrangian fibration is isotrivial, stronger classification results are also available (see~\cite{KLM:Isotrivial}).

%%%%%%%%%%%%%%%%%%

\section{Sheaves on hyper-K\"ahler manifolds}\label{sec:atomic}

On higher dimensional hyper-K\"ahler manifolds, the theory of moduli spaces of sheaves is significantly more complicated than what we saw in Section~\ref{subsec:Examples} in the K3 surface case. 
Here are some bad phenomena one might encounter:
\begin{enumerate}[{\rm (a)}]
    \item For a stable sheaf $F$, the group $\Ext^2(F,F)$ might be large. 
     Consequently, the deformation theory of $F$ requires a much more careful analysis than in the K3 surface setting.
\item A moduli space of stable sheaves might be reducible, non-reduced, and possibly disconnected, even when fixing the topological invariants.
    \item Even if a connected component of a moduli space is smooth, it might not admit a symplectic form.
\end{enumerate}
We will describe explicit examples of these pathologies in Section \ref{subsec:ModuliSpaces}.
However, there are important, distinguished classes of sheaves for which we at least know that a symplectic form exists on the smooth locus of their moduli spaces.
Namely, line bundles on Lagrangian subvarieties, and projectively hyperholomorphic bundles. These two cases were studied independently in the 1990s by Donagi--Markman \cite{DM} and Verbitsky \cite{Verbitsky:Hyperholomorphic}. 

In this section, we review these foundational examples and their associated symplectic forms. 
We then explain how these two seemingly distinct classes of sheaves can be elegantly unified via the recent notion of an atomic (or cohomologically 1-obstructed) object. 
We provide several examples of atomic objects and discuss their moduli spaces, culminating in the modular construction of the OG10 geometry on hyper-Kähler fourfolds.

A highly relevant class of sheaves in this context is that of modular sheaves, which we will encounter but not discuss in full detail. 
They can be roughly thought of as projectively hyperholomorphic (see Definition \ref{defi:Hyperholomorphic}) with respect to every K\"ahler class.
For a comprehensive overview on this topic, we refer to \cite{OG:Survey}.

We work with hyper-K\"ahler manifolds.
We will mostly consider the projective setting, but briefly discuss some analytic concepts as well (e.g., in Section~\ref{subsec:Hyperholomorphic}); there is no known extension yet of the theory to the singular case.

\subsection{Lagrangian subvarieties}\label{subsec:AtomicLagrangians}

The first case we will address is that of the \emph{higher dimensional Beauville--Mukai systems} parameterizing line bundles on deformations of a Lagrangian submanifold. 
Let $X$ be a projective hyper-K\"ahler manifold of dimension $2n$.
Let $i\colon Z\hookrightarrow X$ be a smooth Lagrangian subvariety, and let $L\in\Pic(Z)$ be a line bundle. 
The sheaf $i_*L$ has been extensively studied in~\cite{DM}, where it is shown that $\Def(i_*L)$ is smooth and symplectic.

In fact, Voisin~\cite{Voisin:Lagrangians} showed, via the $T^1$-lifting property, that the Hilbert scheme $\mathrm{Hilb}_X^P$ parametrizing deformations of $Z$ in $X$ is smooth at the point $[Z]$.
Therefore, the (connected) open locus $U \subset \mathrm{Hilb}_X^P$ parameterizing \emph{smooth} deformations of $Z$ is smooth itself. 
Let
\[
M_U\coloneqq\Pic^{0}(\cZ/U)\longrightarrow U
\]
be the identity component of the relative Picard scheme of the universal family $\cZ \subset X \times U$. 
The following was proved in~\cite[Theorem 8.1]{DM}:

\begin{theo}[Donagi--Markman]\label{thm:DM}
There is a canonical symplectic form on $M_U$, which depends only on a polarization on $X$, such that the support morphism $f_U\colon M\to U$ is a Lagrangian fibration. 
\end{theo}

The following question, raised in \cite{DM}, is still open in general. 

\begin{ques}
Does $f_U\colon M_U \to U$ admit a (possibly singular) hyper-K\"ahler compactification $M$ with a Lagrangian fibration $f\colon M\to B$? 
Can we give a modular description of $M$?  
\end{ques}

The first question is exactly the context where Theorem~\ref{thm:Sacca} should apply, once we can extend Theorem~\ref{thm:DM} to a `larger' open subset (in fact, if $\cZ$ comes from a Lagrangian fibration with integral fibers, this is~\cite[Corollary 6.5]{Sac:Lagrangian}).
In the examples discussed in Section~\ref{subsec:ExamplesLagrangian}, the symplectic form of Theorem~\ref{thm:DM} does coincide with the Laza--Sacc\`a--Voisin $2$-form and this is indeed the case.
The modular description, however, is much more complicated.
Only in one case, the variety of lines on a cubic fourfold, a complete answer is known (see Theorem~\ref{thm:Bottini} below).

The symplectic form is easily described at the level of tangent spaces, generalizing the case of curves on K3 surfaces. 
In fact, since $i_*L$ is a stable sheaf, the tangent space to $M_U$ at a point $i_*L$ is given by $\Ext^1(i_*L,i_*L)$ according to standard deformation theory. 
It sits in a short exact sequence 
\[
0 \longrightarrow H^1(Z,\cO_Z) \longrightarrow \Ext^1(i_*L,i_*L) \longrightarrow H^0(Z,\Omega^1_Z) \longrightarrow 0,
\]
where the last map is the differential of $f_U$ at the point $[i_*L]$. 
Of course, this sequence splits, but the splitting is not canonical unless the line bundle $L$ extends uniquely to $X$.
Still, as a first approximation, we construct a symplectic form via the identification $\Ext^1(i_*L,i_*L) \cong H^1(Z,\C)$.
See Example \ref{ex:SympFormLagrangians} for the precise shape of the symplectic form on $M_U$.

\begin{rema}\label{rem:Multiplicativity}
In many cases, the local-to-global spectral sequence 
\begin{equation}\label{eq:Local-to-Global}
E_2^{p,q} = H^{p}(Z,\Omega^q_Z) \implies \Ext^{p+q}(i_*L,i_*L)
\end{equation}
degenerates at $E_2$.
This happens, for instance, if $L \in \Pic^0(Z)$ or if $\dim X = 4$, see \cite{Mladenov:Degeneration} for these results and even more general conditions.  
This gives a (non-canonical) linear isomorphism of graded vector spaces
\[
\Ext^*(i_*L,i_*L) \cong H^*(Z,\C)
\]
which is \emph{not multiplicative}\footnote{This was pointed out to us by Riccardo Carini.} in general, despite this being often claimed in the literature, e.g.,~\cite[Proposition 7.6]{Beckmann:Atomic}.
However, the filtration induced by~\eqref{eq:Local-to-Global} on the algebra $\Ext^*(i_*L,i_*L)$ is compatible with the Yoneda pairing, and the degeneration implies that the associated graded algebra is isomorphic to $H^*(Z,\C)$.
Knowing this suffices for most purposes.
\end{rema}

Choosing a K\"ahler class $\omega \in H^{1,1}(X)$, we can define a symplectic form on $H^1(Z,\C)$ by
\begin{equation}\label{eq:SympFormNaive}
H^1(Z,\C) \times H^1(Z,\C) \longrightarrow \C, \qquad (\alpha,\beta) \longmapsto \int_Z\omega|_Z^{n-1}\cup\alpha \cup \beta. 
\end{equation}
This is skew-symmetric because the cup product is graded commutative and nondegenerate by Poincar\'e duality and the hard Lefschetz theorem.
We see that there are two main differences compared to the case of K3 surfaces:
\begin{itemize}
    \item The symplectic form depends on the choice of a K\"ahler class $\omega\in H^{1,1}(X)$; more generally, we can also choose a class with the hard Lefschetz property. 
    \item Proving that the form is symplectic requires the hard Lefschetz theorem in addition to Serre duality.
\end{itemize}

Since our goal is to obtain a hyper-K\"ahler manifold, we would like the symplectic form to be unique. 
Therefore, it is natural to consider atomic Lagrangians. 

\begin{defi}\label{def:AtomicLagrangians}
A smooth Langrangian subvariety $i \colon Z \hookrightarrow X$ is called {\sf atomic} if 
\begin{enumerate}[{\rm (i)}]
    \item\label{enum:AtomicLagrangians1} The restriction $i^*\colon H^2(X,\Q)\to H^2(Z,\Q)$ has rank $1$.
    \item\label{enum:AtomicLagrangians2} The first Chern class lies in the image $c_1(Z) \in \operatorname{Im}(i^*)$.
\end{enumerate}
\end{defi}

The relation between the second condition and the symplectic form on $M_U$ will be clearer in Section \ref{subsec:ModuliSpaces}.
As we will recall in Section \ref{subsection:defo}, this definition was originally motivated by deformation theory (specifically, deformation theory of the pair $(X,Z)$) and not the uniqueness of the symplectic form.
We should also point out that, while it seems natural to believe that under the atomicity condition there should only be one $2$-form on $M_U$, or better a hyper-K\"ahler compactification, no general result in this direction is known.

\subsection{Hyperholomorphic bundles}\label{subsec:Hyperholomorphic}

Let $X$ be a compact analytic hyper-K\"ahler manifold.
In this section we freely use several analytic concepts (see, e.g.,~\cite{Kob:Stable}).
Another class of sheaves for which we can describe a symplectic form is that of \emph{hyperholomorphic bundles}. 
They were introduced by Verbitsky in~\cite{Verbitsky:Hyperholomorphic}, and the definition depends on the choice of a K\"ahler class $\omega \in H^{1,1}(X)$.

Recall that, to the class $\omega$, one can associate a complex manifold (which is not K\"ahler) with a holomorphic fibration $T_{\omega}(X) \to \P_{\omega}^{1}$ called the {\sf twistor line}.
The base of the fibration $\P_\omega^{1}$ parametrizes complex structures on $X$, which are naturally associated to $\omega$ after choosing a Ricci-flat representative (this can always be achieved thanks to Yau's theorem).
In fact, to a hyper-K\"ahler metric one can associate an action $\H \subset \End(T_{\R}X)$ of the algebra of the quaternions on the real tangent bundle of $X$.
The space $\P_{\omega}^1$ can be identified with the sphere of purely imaginary quaternions with unit norm, each of which defines a complex structure. 

\begin{defi}\label{defi:Hyperholomorphic}
A vector bundle $E$ on $X$ is called $\omega$-{\sf hyperholomorphic} if there exists a hermitian metric $h$ on $E$ such that the Chern connection associated to $h$ is compatible with all the complex structures in $\P_{\omega}^1$. 
\end{defi}

A hyperholomorphic bundle $E$ deforms along the twistor line $T_{\omega}(X)$, but the converse is not true. 
Indeed, in Definition~\ref{defi:Hyperholomorphic} we ask for the existence of the metric $h$, which requires $E$ to be $\mu_{\omega}$-polystable by the Donaldson–Uhlenbeck–Yau Theorem.  
However, a bundle $E$ can deform along the twistor line $T_{\omega}(X)$ without being polystable, for instance if it admits a non-hermitian autodual connection, cf.~\cite{KaledinVerbitsky:Autodual}.

An immediate consequence of the definition is that all tensor operations preserve hyperholomorphic bundles. 
Since $X$ has a Ricci-flat metric, the tangent bundle $T_X$ is stable, and thus it is hyperholomorphic, along with all its tensor powers.
However, finding more examples is a challenging problem, that was open for many years, and was really only solved with the introduction of modular sheaves by O'Grady \cite{OG:Modular}.
The following characterization (see~\cite{Verbitsky:Hyperholomorphic}) is extremely useful to find examples: A vector bundle $E$ on $X$ is $\omega$-hyperholomorphic if and only if it is $\mu_{\omega}$-polystable, and for every stable summand $F \subset E$ we have that $c_1(F) = 0$ and $c_2(F)$ stays of type $(2,2)$ along the twistor line.

If the K\"ahler class $\omega$ is general enough, the general fiber of the twistor line has trivial Picard group. 
This forces the first Chern class of any $\omega$-hyperholomorphic vector bundle to be trivial. 
To circumvent this restriction, which is too strong in practice, we often work with projectively hyperholomorphic bundles.

\begin{defi}
A vector bundle $E$ on $X$ is {\sf projectively hyperholomorphic} if it is $\mu_{\omega}$-polystable, and for every stable summand $F \subset E$ the endomorphism sheaf $\mathscr{E}nd(F)$ is hyperholomorphic.
\end{defi}

We now explain how Verbitsky constructs the symplectic form on $\Ext^1(E,E)$. 
Fix a stable projectively hyperholomorphic vector bundle $E$, and a symplectic form $\sigma \in H^{0}(X,\Omega^2_X)$ on the underlying hyper-K\"ahler manifold.
The natural map $\cO_X \to \mathscr{E}nd(E)$ induces maps in cohomology $H^i(X,\cO_X) \to \Ext^i(E,E)$.
If $\overline{\sigma} \in H^2(X,\cO_X)$ denotes the Hodge conjugate of the symplectic form $\sigma$, we get natural maps 
\[
\Ext^i(E,E) \longrightarrow \Ext^{i+2}(E,E), \qquad \alpha \longmapsto \overline{\sigma} \circ \alpha.
\]
The following result (see~\cite[Theorem 4.2]{Verbitsky:Hyperholomorphic}) is an analogue of the hard Lefschetz theorem, for the operator given by multiplication by $\overline{\sigma}$.
It plays a crucial role in constructing the symplectic form on moduli spaces. 

\begin{theo}[Verbitsky]\label{thm:VerHardLefschetz}
Let $E$ be a projectively hyperholomorphic vector bundle on a hyper-K\"ahler manifold $X$ of dimension $2n$. 
Then, the map
\[
\Ext^i(E,E) \xlongrightarrow{\cong} \Ext^{2n-i}(E,E),\qquad \alpha \longmapsto \overline{\sigma}^{n-1} \circ \alpha 
\]
is an isomorphism.
\end{theo}

Recall that there are trace morphisms 
\[
\Tr\colon\Ext^i(E,E)\longrightarrow H^i(X,\cO_X),   
\]
defined by taking cohomology of the usual trace $\mathscr{E}nd(E) \to \cO_X$. 
They satisfy a linearity relation:
\begin{equation*}\label{eq:LinearityTraces}
\Tr(\eta \circ \alpha) = \eta \cup \Tr(\alpha),\quad\text{ for all }\eta \in H^i(X,\cO_X), \alpha \in \Ext^j(E,E).
\end{equation*}

\begin{coro}
The pairing 
\begin{equation}\label{eq:SympFormHyperholomorphic}
\Ext^1(E,E)\times\Ext^1(E,E)\longrightarrow H^2(X,\cO_X)\cong\C,\qquad (\alpha,\beta) \longmapsto\Tr(\alpha \circ \beta)
\end{equation}
is skew-symmetric and nondegenerate. 
\end{coro}

\begin{proof}
    The skew-symmetry is a simple consequence of the fact that $\Tr([\alpha,\beta]) = 0$, where $\alpha,\beta \in \Ext^*(E,E)$ and $[\alpha,\beta]$ denotes the graded bracket. 
    In turn, the equality $\Tr([\alpha,\beta]) = 0$ follows from the familiar statement that the trace of a bracket vanishes at the level of sheaves. 

    On the other hand, cupping with $\overline{\sigma}$ has the hard Lefschetz property on the algebra $H^{\bullet}(X,\cO_X)$, in other words: $\overline{\sigma}^{n-i} \colon H^i(X,\cO_X) \to H^{2n-i}(X,\cO_X)$ is an isomorphism. 
    Therefore, \eqref{eq:SympFormHyperholomorphic} is nondegenerate if and only if 
    \[
        \overline{\sigma}^{n-1} \cup\Tr(a \circ b) = \Tr(\overline{\sigma}^{n-1} \circ a \circ b)
    \]
    is nondegenerate. 
    This follows from Theorem \ref{thm:VerHardLefschetz} and Serre duality. 
\end{proof}

\subsection{Atomic objects}\label{sec:AtomicObjects}

Atomic (or cohomologically $1$-obstructed) objects were introduced independently by Beckmann \cite{Beckmann:Atomic} and Markman \cite{Markman:Stable} to produce more examples of modular bundles (introduced by O'Grady in~\cite{OG:Modular}) and hyperholomorphic bundles. 
This notion strengthens that of a projectively hyperholomorphic bundle by requiring it to be invariant under derived equivalences.

There are several equivalent ways to approach the definition.
The first approach aims at isolating the objects in $\Db(X)$ for which there is a well-defined Mukai vector, with values in a small dimensional vector space, thus mimicking the Mukai Hodge structure in the case of K3 surfaces.
Following Verbitsky \cite{Verbitsky:Thesis} and Looijenga--Lunts \cite{LL:LieAlgebra}, we define the {\sf extended Mukai lattice}
\[
\widetilde{H}(X,\Q) = H^2(X,\Q) \oplus \Q \mathbf{e} \oplus \Q \mathbf{f},    
\]
equipped with the quadratic form $\tilde{q}$, which restricts to the Beauville--Bogomolov--Fujiki form $q_X$ on $H^2(X,\Q)$, and $\mathbf{e},\mathbf{f}$ are isotropic vectors with $q(\mathbf{e},\mathbf{f}) = -1$.  
We denote an element in $\widetilde{H}(X,\Q)$ either as $r\mathbf{e} + \Delta + s \mathbf{f} $ or, in accordance with the notation for K3 surfaces, as $(r,\Delta,s)$. 
We endow it with the Hodge structure determined by $\widetilde{H}(X)^{2,0} = H(X)^{2,0}$ and compatibility with $q$. 
In other words, we simply declare the two classes $\mathbf{e},\mathbf{f}$ to be of type $(1,1)$.
This space, called the {\sf extended Mukai structure}, serves as the ambient space for the extended Mukai vector

To relate this back to the usual Mukai vector, we recall that there is an isomorphism
\[
\mathfrak{so}(\widetilde{H}(X,\Q),q) \cong \mathfrak{g}(X),
\]
where $\mathfrak{g}(X) \subset \End(H^*(X,\Q))$ is the LLV algebra of $X$. 
It is defined as the sub-Lie algebra of $\End(H^*(X,\Q))$ generated by all the $\mathfrak{sl}_2$-triples associated with all possible K\"ahler classes on all possible complex structures on $X$. 

A foundational result by Taelman \cite{Taelman} establishes that the LLV algebra, together with its actions on $H^*(X,\Q)$ and $\widetilde{H}(X,\Q)$, is an invariant of the derived category.

\begin{theo}[Taelman]\label{thm:Taelman}
Let $\Phi \colon \Db(X) \cong \Db(Y)$ a derived equivalence between two hyper-K\"ahler varieties. Then:
\begin{enumerate}[{\rm (a)}]
        \item There exists an isomorphism of lie algebras $\Phi^{\mathfrak{g}} \colon \mathfrak{g}(X) \xlongrightarrow{\sim} \mathfrak{g}(Y)$ making the morphism $\Phi^{H} \colon H^*(X,\Q) \xlongrightarrow{\sim} H^*(Y,\Q)$ equivariant. 
        \item here exists a Hodge isometry 
        \[
          \Phi^{\widetilde{H}} \colon \widetilde{H}(X,\Q) \xlongrightarrow{\sim}\widetilde{H}(Y,\Q) 
        \]
        which makes $ \Phi^\mathfrak{g}$ equivariant. 
\end{enumerate}
\end{theo}

This theorem, albeit abstractly, provides a link between the full cohomology $H^{*}(X,\Q)$ and $\widetilde{H}(X,\Q)$, which we can use to make sense of an extended Mukai vector. 

\begin{defi}
An object $E \in \Db(X)$ is called {\sf atomic}, if there exists an element $\tilde{v}(E) \in \widetilde{H}(X,\Q)$, called the {\sf extended Mukai vector}, such that 
    \[
        \mathrm{Ann}(v(E)) = \mathrm{Ann}(\tilde{v}(E))
    \]
as sub-Lie algebras of $\mathfrak{g}(X)$. 
\end{defi}

\begin{rema}\label{rem:Normalization}
    The extended Mukai vector is only well defined up to constants. 
    By \cite[Proposition 3.8]{Beckmann:Atomic} and \cite[Theorem 6.13(3)]{Markman:Stable}, if $r(E) \neq 0$, then it can be normalized as 
    \[
        \tilde{v}(E) = \left(r(E),c_1(E),s(E)\right),
    \]
    for some $s(E) \in \Q$. 
    If $r(E) = 0$, it is not yet clear which normalization one should choose.  
\end{rema}

An immediate and powerful consequence of Taelman's theorem is that the image of an atomic object under a derived equivalence $\Phi$ remains atomic, with extended Mukai vector $\tilde{v}(\Phi(E)) = \Phi^{\tilde{H}}(\tilde{v}(E))$.

In higher dimensions, the atomicity condition severely restricts the classical Mukai vector.
To understand more explicitly how the extended Mukai vector translates back to the classical one, we recall another theorem of Verbitsky.
First we introduce the {\sf Verbitsky component} $\mathrm{SH}(X,\Q) \subset H^*(X,\Q)$, defined as the subalgebra generated by $H^2(X,\Q)$.
It forms an irreducible representation of the LLV algebra $\mathfrak{g}(X)$. 
Verbitsky~\cite{Verbitsky:Thesis} showed (see also \cite{Taelman}) that there is a $\mathfrak{g}(X)$-equivariant short exact sequence
\begin{equation*}
0 \longrightarrow \mathrm{SH}(X,\Q) \longrightarrow \Sym^n(\widetilde{H}(X,\Q)) \longrightarrow \Sym^{n-2}(\widetilde{H}(X,\Q)) \longrightarrow 0.
\end{equation*}
By equipping $\Sym^n(\widetilde{H}(X,\Q))$ with the natural bilinear pairing induced by $\tilde{q}$, the first map is an isometry up to scaling \cite[Proposition 3.8]{Taelman}.
Denoting by
\[
T\colon\Sym^n(\widetilde{H}(X,\Q))\longrightarrow\mathrm{SH}(X,\Q)
\]
the orthogonal projection, we have that 
\[
    T\big(\tilde{v}(E)^{(n)}\big) = v(E)_{\mathrm{SH}},
\]
for any atomic object $E \in \Db(X)$, see \cite[Proposition 3.3]{Beckmann:Atomic} and \cite[Theorem 6.13(2)]{Markman:Stable}. 
Here $v(E)_{\mathrm{SH}}$ denotes the orthogonal projection of $v(E)$ onto $\mathrm{SH}(X,\Q)$ with respect to the Mukai pairing.
For hyper-K\"ahler manifolds of $\mathrm{K3}^{[2]}$-type, where $\mathrm{SH}(X,\Q) = H^*(X,\Q)$, this gives the explicit formula
\[
v(E) = r(E) + c_1(E) + \frac{1}{2r(E)}\big(c_1(E)^2 - \tilde{q}(\tilde{v}(E),\tilde{v}(E))\mathsf q_2\big) + \frac{s(E)}{r(E)}\lambda^{\vee} + \frac{s(E)^2}{2r(E)}\mathsf{pt},
\]
if $r(E)\neq0$, where $\mathsf{q}_2= \frac{1}{25}\mathsf{q}^{\vee}\in H^4(X,\Q)$.

\subsection{The obstruction map}\label{subsection:defo}

An alternative motivation for the notion of atomic objects comes from deformation theory. 
If $X$ is a smooth projective variety, we set 
\[
    \mathrm{HT}^2(X) = H^2(X,\cO_X) \oplus H^1(X,T_X) \oplus H^0(X,\wedge^2T_X).
\]
By the Hochschild--Konstant--Rosenberg isomorphism, this space is isomorphic to the second Hochschild cohomology $\mathrm{HH}^2(X)$, and it controls the first order noncommutative deformations of $X$ (we will discuss more about this in Section~\ref{sec:categories}).
Specifically: $H^1(X,T_X)$ parametrizes the classical deformations, $ H^0(X,\wedge^2T_X)$ the deformations of the product on the structure sheaf $\cO_X$, and $H^0(X,\cO_X)$ the gerby deformations.
For any deformation $\cD_{\eta}$ parametrized by $\eta \in \mathrm{HT}^2(X)$, and any object $E \in \Db(X)$, the obstruction to lifting $E$ to an object of $\cD_{\eta}$ is controlled by the \emph{obstruction map}:
\begin{equation*}
\mathrm{obs}_E\colon\mathrm{HH}^2(X) \longrightarrow \Ext^2(E,E), \qquad \eta \longmapsto \eta_E,
\end{equation*}
where, if we think of $\eta \in \mathrm{HH}^2(X)$ as a natural transformation $\eta_E \colon \id_E \to \id_E[2]$, then $\eta_E$ is the evaluation at $E$. 
By \cite{Toda:Deformations}, we can lift an object $E \in \Db(X)$ to one in $\cD_{\eta}$ if and only if $\mathrm{obs}_E(\eta) = 0$. 

For any object $E \in \Db(X)$, the \emph{cohomological obstruction map} is 
\[
\mathrm{obs}_E^H\colon\mathrm{HT}^2(X)\longrightarrow H^{*}(X,\C), \qquad \eta \longmapsto \eta \lrcorner \ch(E).
\]
Roughly speaking, the cohomological obstruction map controls whether the Chern character $\ch(E)$ remains of Hodge type along the first order deformation given by $\eta$. 
One can completely characterize the atomic objects in terms of the obstruction map (see~\cite[Theorem 1.2]{Beckmann:Atomic}).

\begin{theo}[Beckmann]\label{thm:Atomic1Obstructed}
An object $E \in \Db(X)$ is atomic if and only if the cohomological obstruction map has rank $1$. 
\end{theo}

We say that an object $E$ is (cohomologically) $1$-obstructed if the rank of its (cohomological) obstruction map is $1$.
For a sheaf $E$ on $X$ of rank $r(E)\neq0$, the image of $H^2(X,\cO_X)$ via the cohomological obstruction map is always nonzero.
In particular, the cohomological obstruction map has at least rank $1$. 
So we can think of the atomicity condition as saying that the Chern character deforms maximally with respect to all noncommutative deformations of $X$. 

As expected from the definitions, if an object is $1$-obstructed it is also cohomologically $1$-obstructed~\cite[Theorem 1.3]{Beckmann:Atomic}.
Beckmann conjectures~\cite[Conjecture A \& Corollary 4.7]{Beckmann:Atomic} that for simple atomic objects the reverse implication also holds. 

\begin{conj}[Beckmann]\label{Con:Atomic1Obs}
Let $E \in \Db(X)$ be a simple atomic object. Then $E$ is $1$-obstructed.
\end{conj}

\subsection{Examples of atomic objects}\label{subsec:ExamplesAtomic}
We now review several key examples of atomic objects, starting with the simplest cases before returning to Lagrangian subvarieties and hyperholomorphic bundles later on. 

\subsubsection*{First examples}
An immediate source of $1$-obstructed objects is given by $\P^n$-objects (introduced in \cite{HT:Pn}), since their $\Ext^2$ is one-dimensional.
Recall that an object $E \in \Db(X)$ is a  $\P^n$-{\sf object} if there is an isomorphism of graded algebras $\Ext^{*}(E,E) \cong H^{*}(\P^n,\C)$\footnote{Note that here we are assuming that $X$ has trivial canonical bundle. Otherwise, one should also require compatibility with the Serre functor.}. 

The simplest $\P^n$-objects are the line bundles.  
For the structure sheaf $\cO_X$, the extended Mukai vector can be computed to be
\[
\tilde{v}(\cO_X)= \left(1,0,r_X\right),
\]
consistently with Remark~\ref{rem:Normalization}.
Here $r_X$ is a structure constant, for whose definition we refer to \cite{Beckmann:Extended}, which can be computed to be 
\[
r_X = 
\begin{cases}
\frac{n + 3}{4},&\text{for }\mathrm{K3^{[n]}}\text{ and }\mathrm{OG10}\\ 
\frac{n+1}{4},&\text{for }\mathrm{Kum_n}\text{ and }\mathrm{OG6}.
\end{cases}
\]
For a more general line bundle $L \in \Pic(X)$, the extended Mukai vector is 
\[
\tilde{v}(L)=\left(1,c_1(L),\textstyle\frac{1}{2}q(c_1(L)) + r_X\right). 
\]
The computation goes through Theorem~\ref{thm:Taelman}, and the computation of the action of the autoequivalence $-\otimes L$ on $\widetilde{H}(X,\Q)$, given by the familiar formula
\[
B_L(r,\Delta,s)=\left(r,\Delta + rc_1(L),s+q(\Delta,L) + r\textstyle \frac{1}{2}q(c_1(L))^2\right).
\]

Skyscraper sheaves of point $\cO_x$ are another source of atomic objects.
Since their Mukai vector is $v(\cO_x) = \mathrm{pt} \in H^{4n}(X)$, it is immediate to see that 
\[
\ker\mathrm{obs}^H_{\cO_x}=H^2(X,\cO_X)\oplus H^1(X,T_X) 
\] 
So that $\cO_x$ is atomic, and its extended Mukai vector is $\tilde{v}(\cO_x) = (0,0,1)$.
In fact, one can also check that $\cO_x$ is actually $1$-obstructed.

\subsubsection*{Lagrangian subvarieties}
%%Markman Thm 1.1, rank 0 atomic sheaves are supported on Lagrangians or points.
One of the main results of \cite{Beckmann:Atomic,Markman:Stable} is the following.
It relates the atomicity of $\cO_Z$ to the atomicity of $Z \subset X$, in the sense of Definition \ref{def:AtomicLagrangians}.  
Let $i\colon Z\hookrightarrow X$ be a smooth Lagrangian subvariety. 
Then $\cO_Z$ is atomic if and only if $Z$ is atomic.
Moreover, if $\lambda$ is a generator of ${\ker i^*}^{\perp}$, where $i^*\colon \NS(X) \to \NS(Z)$ is the restriction map, and we write $c_1(Z) = ti^*\lambda$ for some $t \in \Q$, the extended Mukai vector is 
\[
\tilde{v}(\cO_Z) = \left(0,\lambda,\textstyle \frac{1}{2}tq(\lambda)\right).
\]

We can now better explain the conditions in Definition \ref{def:AtomicLagrangians}, in light of Theorem \ref{thm:Atomic1Obstructed}.
We denote by $\mathrm{Def}(X)$ the deformation space of $X$, and by $\Def(X,Z)$ the deformation space of the pair $(X,Z)$.
There is a natural forgetful map $\Def(X,Z) \to \Def(X)$, whose image is contained in the locus $\Def(X,[Z])$ where the class $[Z]$ remains of Hodge type. 
The tangent space at the distinguished point of $\Def(X)$ is identified with $H^1(X,T_X)$ via the Kodaira--Spencer map, and it can be further identified with $H^{1,1}(X)$ via the isomorphism $T_X\cong\Omega^1_X$ given by the symplectic form. 
The fundamental result of Voisin~\cite{Voisin:Lagrangians} tells us that the map 
\[
\mathrm{Def}(X,Z)\longrightarrow\mathrm{Def}(X,[Z])
\]
is surjective and $\mathrm{Def}(X,[Z])$ is smooth.
Its tangent space is identified with the kernel of the restriction map
\[
\ker i^* \subset H^{1,1}(X).
\]

The restriction map $i^*$ has at least rank $1$, as a K\"ahler class restricts to a K\"ahler class.
So, condition~\ref{enum:AtomicLagrangians1} requires that $Z$ (or equivalently $[Z]$) deform maximally along \emph{commutative} deformations of $X$.
Condition~\ref{enum:AtomicLagrangians2} deals with the noncommutative deformations, as explained in~\cite{Beckmann:Atomic}.

\begin{rema}
    If we decompose the second integral cohomology according to its algebraic and transcendental part $H^2(X,\Q) = \NS(X)_{\Q} \oplus T(X)_{\Q}$, we see that the restriction morphism $i^*$ vanishes on the transcendental $T(X)_{\Q}$. 
    This follows from the fact that $T(X)_{\Q}$ is the smallest sub Hodge structure of $H^2(X,\Q)$ whose complexification contains the symplectic form $\sigma$.
    Therefore, only the Néron--Severi group contributes to the image of $i^*$.
\end{rema}

\begin{exam}\label{ex:AtomicLagrangians}
Examples of atomic Lagrangians include:
\begin{enumerate}[{\rm (a)}]
    \item Projective spaces $\P^n \subset X$, as they are $\P^n$-objects. 
    \item Smooth fibers of a Lagrangian fibration $\pi \colon X \to \P^n$. In this case, $\tilde{v}(\cO_Z) = (0,f,0)$, where $f =c_1(\pi^*(\cO_{\P^n}(1)))$.
    \item For any smooth cubic fourfold $W \subset \P^5$, as we mentioned before, the variety $F(W)$ of lines contained in $W$ is a smooth hyper-K\"ahler fourfold of K3$^{[2]}$-type. The surface $F(W_H) \subset F(W)$ of lines contained in a smooth hyperplane section is a smooth atomic Lagrangian subvariety, see \cite[Section 8.1.3]{Beckmann:Atomic}. 
    Its extended Mukai vector is 
    \[
        \tilde{v}(\cO_{F(W_H)}) = (0,h,-3),
    \]
    where $h$ is the class of the Pl\"ucker polarization on $F(W)$.
    \item A double EPW sextic $X$ is a hyper-K\"ahler fourfold equipped with $2:1$ morphism to a sextic hypersurface in $\P^5$ (\cite{OG:EPW}).
    The sextic is singular along a smooth surface if $X$ is general enough, and the covering involution fixes an isomorphic copy of this surface $Z \subset X$. 
    This surface is atomic and rigid, see \cite[Section 8.1.4]{Beckmann:Atomic}.
    Its extended Mukai vector is 
    \[
        \tilde{v}(\cO_Z) = (0,h,-3),
    \]
    where $h$ is the polarization on $X$.
\end{enumerate}
\end{exam}

Not all smooth Lagrangian subvarieties are atomic, here are some examples.

\begin{exam}
Let $S$ be a K3 surface and let $C\subset S$ be a smooth curve of genus $g(C)>1$. 
Then the symmetric product $i\colon C^{(2)} \subset S^{[2]}$ is a Lagrangian surface, which is not atomic since the restriction $i^*$ has rank $\geq2$ (by taking the restrictions of the exceptional divisor, which restricts to the diagonal, and of an ample class).
\end{exam}

\begin{exam}
Let $(S,H)$ be a polarized K3 surface of genus $g\geq2$, with $\Pic(S)=\Z\cdot H$.
Let $r,d$ be two coprime integers with $r \geq 2$, and consider the Beauville--Mukai system
\[
M \coloneqq M_{S,H}(\bv)\to|rH|,\qquad \bv=(0,rH,d+r(1-g)).
\]
For a smooth curve $C\in|H|$, the fiber over the nonreduced curve $rC$ has multiple components, one of which is the moduli space $N\coloneqq \cU_{C}(r,d)$ of stable bundles of rank $r$ and degree $d$ supported on $C$. 
It is a smooth Lagrangian subvariety $i\colon N\hookrightarrow M$ and the restriction $i^*$ has rank $1$; this follows easily from the fact that the class $f\in\NS(M)$ corresponding to the Lagrangian fibration restricts trivially to $N$, since $N$ is contained in a fiber.
However, $\omega_N \not\in \Im i^* $. 
Indeed, there is a commutative diagram 
\[\begin{tikzcd}
	{} & {\widetilde{H}(S,\Z)_{\mathrm{alg}} \supset\bv^{\perp}} & {\NS(M)} \\
	& {H^{\mathrm{ev}}(C,\Z) \supset (r,d+r(1-g))^{\perp}} & {\NS(N)}
	\arrow["{\lambda_S}", from=1-2, to=1-3]
	\arrow[from=1-2, to=2-2]
	\arrow["{i^*}", from=1-3, to=2-3]
	\arrow["{\lambda_C}", from=2-2, to=2-3]
\end{tikzcd}\]
where the vertical morphisms are the restrictions, and the horizontal ones are the Donaldson--Mukai morphisms. 
On the other hand, by \cite[Théorème D]{DN:VectorBundles} we have 
\[
\NS(N) = \Im \lambda_C \oplus p^*\NS(J^d(C)),
\]
where $p\colon N \to J^d(C)$ denotes the determinant map $F \mapsto \det(F)$, and by \cite[Théorème E]{DN:VectorBundles} $\omega_N$ does not lie in the image of $\lambda_C$.  
So, $N$ is an example of a Lagrangian which has restriction map of rank $1$, but it is not atomic.  
\end{exam}

\subsubsection*{Immersed Lagrangians}

Since smooth Lagrangian subvarieties of hyper-K\"ahler manifolds are notoriously difficult to construct, one can allow some singularities while the theory still works 
as in the smooth case.
Recall that, on a projective hyper-K\"ahler manifold $X$, we say that a finite unramified morphism $j\colon Z\to X$ from a smooth projective variety is an {\sf immersed Lagrangian} if its image is a Lagrangian subvariety.
As demonstrated in \cite{GL:Immersed}, the theory of atomicity for immersed Lagrangians is almost perfectly analogous to that of their smooth counterparts
In particular, by~\cite[Theorem 4.10]{GL:Immersed}, we have a criterion for when $j_*\cO_Z$ is atomic in terms of the pull-back map and of the first Chern class of $Z$.

\begin{exam}\label{ex:IM}
Let $Y \subset \P^6$ be a smooth cubic fivefold, and $W\subset Y$ a smooth cubic fourfold. 
We assume that $W$ does not contain any planes.
If $F_2(Y)$ denotes the surface of planes contained in $Y$, the morphism
\[
j \colon F_2(Y) \longrightarrow F(W),\qquad\Pi\longmapsto\Pi\cap W
\]
is well defined.
For general choices it defines an immersed Lagrangian subvariety, by \cite{IM:Cubic}.
In fact, it is the normalization of its image, which has $47061$ isolated singular points, see \cite[Proposition 7]{IM:Cubic}. 
An adjuction formula computation yields $c_1(F_2(Y)) = -3j^*h$, so the extended Mukai vector is
\[
\tilde{v}(j_*\cO_{F_2(Y)}) = (0,h,-9) \in \widetilde{H}(F(W),\Q).
\]
Lastly, we also have $j_*[F_2(Y)] = 63[F(W_H)] \in H^4(F(W),\Z)$, by \cite[Lemma 6]{IM:Cubic}.

A similar situation exists for Gushel--Mukai fourfold/fivefold and the associated double EPW sextic. We will not review this here, and refer to~\cite{OG:EPW,IM:GM,MOS:GM} for more details.
\end{exam}

\subsubsection*{Atomic vector bundles}
The most important result about atomic vector bundles is that they are projectively hyperholomorphic.
More precisely, for a hyper-K\"ahler manifold $X$ and a K\"ahler class $\omega\in H^{1,1}(X)$, if $E$ is $\mu_{\omega}$-polystable atomic vector bundle, then it is $\omega$-projectively hyperholomorphic. 

\begin{rema}
The notion of projectively hyperholomorphic bundles depends on the choice of the K\"ahler form $\omega$, so the reverse implication is not true.
In fact, even if the discriminant $\Delta(E) = c_2(\mathcal{E}nd(E))$ remained of Hodge type with respect to any deformation of $X$, it still would not be atomic. 
The bundle $E$ would be projectively hyperholomorphic with respect to any polarization for which it is stable, and thus deform to any deformation of $X$, but there could be nontrivial obstructions to deform $E$ along the noncommutative deformations of $X$. 
The most striking example of this is the tangent bundle $T_X$, as shown in \cite{Beckmann:Atomic}, but see also \cite{OG:Modular3} for more examples.
\end{rema}

This result is the central motivation behind the paper \cite{Markman:Stable}, where the following strategy to construct examples of hyperholomorphic bundles is proposed: start with atomic objects, transform them into locally free sheaves via derived equivalences, and check slope stability. 
In many cases, the crucial input for the stability is given by the results of \cite{OG:Modular}, especially on the notion of suitable polarization. 

%%Say something about Sym^n and SH, and projection of discriminant
%% Markman's examples, O'Grady's examples, My examples

\begin{exam}%[{{\cite[Theorem 1.4]{OG:Modular}}}]
The first examples we wish to discuss were constructed by O'Grady \cite{OG:Modular,OG:Modular2} on hyper-K\"ahler manifolds of $\mathrm{K3}^{[n]}$-type. They are $\P^n$-objects, so atomicity is automatic.
We discuss them in the $n=2$ case, the general case works similarly with more numerical assumptions.
If $(X,h)$ is a general polarized hyper-K\"ahler manifold of $\mathrm{K3}^{[2]}$-type, and $r$ is a positive integer, under suitable numerical assumptions there exists a unique slope stable vector bundle $E$ on $(X,h)$ such that 
    \[
        r(E) = r^2, \ c_1(E) = rh, \  \Delta(E) = \frac{r(r-1)}{12}c_2(X)
    \]
    and $\Ext^{\bullet}(E,E) \cong H^{\bullet}(\P^n,\C)$.
These bundles are atomic, and their extended Mukai vector can be computed using the formula 
\begin{equation*}
    \Delta(E) = \left(\tilde{q}(\tilde{v}(E)) + 2r_Xr(E)^2\right)\mathsf{q}_2,
\end{equation*}
established in~\cite{Bot:Towards}.
If one specializes $X$ to a Hilbert scheme $S^{[n]}$, for a suitable K3 surface $S$, these bundles are in the orbit of the structure sheaf $\cO_{S^{[n]}}$ under the action of derived autoequivalences of $\Db(S^{[n]})$.
\end{exam}

\begin{exam}
Other examples of atomic locally free sheaves are constructed by Markman \cite[Theorem 1.6]{Markman:Stable}. 
They are obtained by deforming locally free stable sheaves on certain Hilbert schemes $S^{[n]}$, which are in the orbit of the skyscraper sheaf of a point $\cO_x$ under derived equivalences. 
If $E$ is any one of such sheaves, their extended Mukai vector is 
    \[
        v(E) = \left(n!r^n,\lambda,\frac{1}{2n!r^n}q(\lambda)\right),
    \]
    where $\lambda = c_1(E)$, see \cite[Lemma 11.3]{Markman:Stable}.
\end{exam}

\begin{rema}
Since these bundles are projectively hyperholomorphic with respect to an open cone of K"ahler classes, they can be deformed as twisted sheaves on any hyper-K"ahler variety of type $\mathrm{K3}^{[n]}$, as explained, for instance, in \cite[Theorem 3.4]{Markman:Stable}. As such, they give rise to many twisted bundles on a very general hyper-K"ahler variety, and they can be used to study the Period--Index problem. This strategy was successfully applied using Markman's bundles in \cite{HMSYZ:PeriodIndex} and using O'Grady's bundles in \cite{BH:PeriodIndex} to obtain the bounds $\mathrm{ind}(\alpha) \mid \mathrm{per}(\alpha)^{2n}$ and $\mathrm{ind}(\alpha) \mid \mathrm{per}(\alpha)^n$ for ``non-special'' Brauer classes, respectively.
\end{rema}

\begin{exam}
Other examples are constructed in \cite{Bot:Towards,Bot:OG10} by deforming a cubic fourfold $W$ until its variety of lines $F(W)$ acquires a Lagrangian fibration.
Then transforming line bundles supported on $F(W_H)$ into (possibly twisted) locally free sheaves via a derived equivalence induced by the Poincar\'e sheaf. 
A similar approach has been applied to the immersed Lagrangians subvarieties in \cite{GL:Immersed}.
\end{exam}

\subsection{Moduli spaces}\label{subsec:ModuliSpaces}
As anticipated, moduli spaces of semistable sheaves on higher dimensional hyper-K\"ahler manifolds are not as well behaved as those for K3 surfaces in general. 
Here are some explicit examples to highlight these situations.

\begin{exam}
As previously noted, the tangent bundle $T_X$ is slope stable, however
\[
\Ext^2(T_X,T_X) \cong H^2(X,\mathcal{E}nd(T_X)) \cong H^2(X,\Omega_X^{\otimes 2}) \supset H^{2,2}(X),
\]
where the second isomorphism is induced by the symplectic form $T_X \cong \Omega^1_X$.
It follows from Verbitsky's description of the Verbitsky component that $h^{2,2}(X) \geq \binom{b_2(X)-2}{2}$, which, for example, amounts to $\binom{21}{2} = 210$ for varieties of $\mathrm{K3}^{[2]}$-type. Thus $\Ext^2(T_X,T_X)_0$ is far from trivial.

Although one might expect that $\operatorname{Ext}^2(T_X, T_X)_0$ is not the “true" obstruction space and that one should instead consider the kernel of the Buchweitz--Flenner semiregularity map~\cite{BF:Semiregularity}, this kernel is usually still too large to provide useful control. 
For $T_X$ on a $\mathrm{K3}^{[2]}$-type manifold, the target of the semiregularity map $\operatorname{Ext}^2(T_X, T_X) \to \bigoplus_{p=0}^2 H^{p+2}(X, \Omega_X^p)$ is isomorphic to $H^2(X, \mathbb{C})$, so the kernel's dimension remains at least $h^{2,2}(X) - b_2(X)$.
\end{exam}

\begin{exam}
Let $(X,H)$ be a double EPW sextic, where $H$ is the polarization obtained by pull-back from $\P^5$. 
It comes with an involution $\iota\in\Aut(X)$, whose fixed locus is a Lagrangian surface $i \colon Z \hookrightarrow X$ with $q(Z) =0$. 
As explained in \cite[Remark 1.4]{Ferretti:Chow}, there is a $2$-torsion class $\tau \in \NS(Z)$, defined by 
\[
c_1(Z) = -H|_Z +\tau.
\]
If $L \in \Pic(Z)$ is a line bundle with $c_1(L) = \tau$, then $i_*\cO_Z$ and $i_*L$ are stable sheaves on $X$, with respect to any polarization, and $v(i_*L) = v(i_*\cO_Z)$. 
More than that, they also have the same class in the topological K-theory $K_0^{\mathrm{top}}(X)$, as it is torsion-free.
However, since $q(Z) = 0$, they are both rigid, so that they define two disjoint reduced points in the moduli space of stable objects. 
In particular, fixing the Mukai vector (or even the class in $K_0^{\mathrm{top}}(X)$) does not suffice to isolate a connected component.
\end{exam}

\begin{exam}
Assume that $\dim X >2$, and let $H \subset X$ be an ample line bundle. 
Assume that every divisor $D \in |H|$ is integral (this happens e.g.\ if $\rho(X) = 1$). 
Then $H^1(D,\cO_D)=0$ and the morphism
\[
|H|\longrightarrow  M_{X,H}(v(\cO_D)),\qquad D\longmapsto\cO_D
\]
gives the embedding of a connected component, which is not symplectic. 
See~\cite{KRZ:NonSymplectic} for another example of a smooth moduli space which is not symplectic.
\end{exam}

\subsubsection*{The symplectic form}
As we explained before, for projectively hyperholomorphic bundles and line bundles supported on Lagrangian subvariety the $\Ext^1(E,E)$ has a natural symplectic form, which can be nicely described in terms of the obstruction map.
Namely, fix $\eta \in HH^2(X)$ and consider the pairing
\begin{equation}\label{eq:SymplecticFormGeneral}
\Ext^1(E,E)\times\Ext^1(E,E)\longrightarrow\C, \qquad (\alpha,\beta)\longmapsto\mathrm{Tr}\left(\mathrm{obs}_E(\eta)^{n-1} \circ \alpha \circ \beta\right).
\end{equation} 
One can show that it is skew-symmetric because the trace kills the graded brackets.

\begin{exam}[Lagrangian subvarieties]\label{ex:SympFormLagrangians}
In the case of a line bundle on a Lagrangian subvariety $Z \subset X$, treated in Section \ref{sec:Lagrangian}, \eqref{eq:SymplecticFormGeneral} is the symplectic form of Theorem \ref{thm:DM}. 
Because of the issues mentioned in Remark \ref{rem:Multiplicativity}, the relation to \eqref{eq:SympFormNaive} is not straightforward.
However, it is easy to check that the hard Lefschetz theorem implies that it is nondegenerate. 
\end{exam}

Obviously, it is only interesting if $\mathrm{obs}_E(\eta)\neq0$.
If we aim to obtain moduli spaces of sheaves that are hyper-K\"ahler themselves, \eqref{eq:SymplecticFormGeneral} suggests that we should only consider moduli spaces of $1$-obstructed objects.
By Conjecture~\ref{Con:Atomic1Obs} on the simple locus $1$-obstructed is equivalent to atomic, which is a cohomological condition depending only on the Mukai vector.
To have a symplectic moduli space one should prove the following conjecture, which was put forward in \cite[Section 8]{Beckmann:Atomic}.

\begin{conj}[Beckmann]
Let $E$ be a stable atomic sheaf, then any nonzero element in the image of the obstruction map $\mathrm{obs}_E(\eta) \in \Ext^2(E,E)$ has the hard Lefschetz property for the Yoneda product on $\Ext^{\bullet}(E,E)$.
\end{conj}

\subsubsection*{A modular construction of OG10}%\label{sec:OG10}

Despite the general difficulties for higher-dimensional moduli spaces, there is one spectacular case where everything works perfectly: line bundles on the Lagrangian surface $F(W_H) \subset F(W)$ of lines contained in a hyperplane section of a general cubic fourfold $W \subset \mathbb{P}^5$
In fact, if $W_H \coloneqq W \cap H$ is a smooth hyperplane section, any line bundle $L \in \Pic^0(F(W_H))$ defines a smooth point of the moduli space $M_{(F(W),h)}(\mathbf{v})$, where $\mathbf{v} \coloneqq v(\cO_{F(W_H)})$.
The closure of this locus identifies a connected component $M^{\circ}_{(F(W),h)}(\mathbf{v})$.
The following is the main result of {\cite{Bot:OG10}}:

\begin{theo}\label{thm:Bottini}
    Let $W \subset \P^5$ be a general cubic fourfold. 
    Let $h$ be the Pl\"ucker polarization on $F(W)$, and $\mathbf{v} \coloneqq v(\cO_{F(W_H)})$. 
    Denote by $\cF \subset F(W_H) \times |\cO_W(1)|$ the universal surface of lines contained in hyperplane sections. 
    Then, 
    \begin{enumerate}[{\rm (a)}]
        \item The component $M^{\circ}_{(F(W),h)}(\mathbf{v})$ is isomorphic to $\overline{\Pic^0}(\cF/|\cO_W(1)|)$, the relative compactified Picard scheme\footnote{This is the closure of the locally free locus in the relative moduli space of torsion-free sheaves of rank $1$ supported on the fibers of $\cF \to |\cO_W(1)|$.} of $\cF \to |\cO_W(1)|$, in the sense of Altman--Kleiman.
        \item The component $M^{\circ}_{(F(W),h)}(\mathbf{v})$ is a smooth hyper-K\"ahler variety isomorphic to the LSV system $\cJ_W$. The symplectic form is given by \eqref{eq:SymplecticFormGeneral}.
    \end{enumerate}
\end{theo} 

The difficulty lies, of course, in the singular hyperplane sections $W_H$. 
In this case, the surface $F(W_H)$ is a local complete intersection surface which is not normal, and whose singularities can be understood to some extent.
For instance, if $W_H$ has a single ordinary double point, the surface $F(W_H)$ has $A_1$ singularities along a smooth curve of genus $4$, which parametrizes the lines containing the singular point. 
In general, the singularities are more complicated, but one can still show that any sheaf in $\overline{\Pic^0}(\cF/|\cO_W(1)|)$ is Cohen--Macaulay on its support.

If one chooses the cubic $W$ so that $F(W)$ has a Lagrangian fibration $F(W) \to \P^2$, then by a result of Markman \cite[Theorem 7.11]{Markman:Lagrangians} combined with the twisted version\footnote{This construction was recently further generalized in \cite{HSh:TwistedArinkin} to allow for twists on the source as well.
} \cite[Theorem 3.3]{Bot:OG10} of Arinkin's result \cite[Theorem C]{Arinkin}, there exists a derived equivalence
\[
    \Phi_{\cP}\colon \Db(F(W)) \xlongrightarrow{\cong}\Db(M,\theta),\quad\text{ for some } \theta \in \Br(M).
\]
Here $M$ is a Beauville--Mukai system on a K3 surface of degree $2$.
The equivalence $\Phi_\mathcal{P}$ acts on general fibers like a classical Fourier–Mukai duality, mapping skyscraper sheaves to line bundles on the corresponding fibers, and mapping the Cohen–Macaulay sheaves in $M_{(F(W),h)}^\circ(v)$ to locally free (twisted) bundles on $M$. These transformed bundles are slope-stable thanks to O'Grady's stability theory for modular sheaves \cite{OG:Modular}. 
Consequently, $M_{(F(W),h)}^\circ(v)$ is identified precisely with a moduli space parameterizing solely stable, projectively hyperholomorphic twisted bundles.

This observation turns out to be crucial for the smoothness of $M^{\circ}_{(F(W),h)}(\mathbf{v})$, whose proof is quite involved and relies heavily on the results and computations of \cite{LSV:OG10}. 
The proof will be drastically simplified if one could establish the following difficult conjecture (\cite[Conjecture C]{Beckmann:Atomic}), which combined with the formality result of \cite{MO:Formality}, would imply smoothness of $M^{\circ}_{(F(W),h)}(\mathbf{v})$ directly. 

\begin{conj}[Beckmann]
Let $E$ be a stable atomic vector bundle on a hyper-K\"ahler manifold. Then the Yoneda product
\[
\Ext^1(E,E) \times \Ext^1(E,E) \longrightarrow \Ext^2(E,E)
\]
is skew-symmetric.
\end{conj}

To successfully run the argument above we need to extend the definition of atomic objects to include the twisted ones as well. 
This is partly done in \cite{Bot:OG10}, however a fundamental question is still open. 
Namely: is the LLV algebra invariant under derived equivalences of twisted hyper-K\"ahler varieties?

\begin{ques}
Let $X,Y$ be two hyper-K\"ahler varities, and $\alpha \in \Br(X),\beta \in \Br(Y)$ be two topologically trivial Brauer classes.
Assume that there is a derived equivalence $\Phi \colon \Db(X,\alpha)\xlongrightarrow{\sim}\Db(Y,\beta)$, is there an isomorphism of Lie algebras $\Phi^{\mathfrak{g}} \colon \mathfrak{g}(X)\xlongrightarrow{\sim} \mathfrak{g}(Y)$?
\end{ques}

\subsubsection*{The LLX example}
One might wonder if there is a similar modular interpretation for the $42$-dimensional irreducible holomorphic symplectic variety $M_{\mathrm{LLX}}$ constructed in \cite{LLX:Lagrangian} and reviewed briefly in Section~\ref{subsec:ExamplesLagrangian}.
With the notation of Example~\ref{ex:IM}, we see that (see also \cite{BC:Semirigid} for these computations) up to tensoring by $\cO_W(h)$, we have
\[
\tilde{v}\left(j_*\cO_{F_2(Y)}(h)\right) = 63\,\mathbf{v},
\]
where we keep the above notation that $\mathbf{v} = v(\cO_{F_2(Y)})$. 
In other words, the variety $M_{\mathrm{LLX}}$ is birational to an irreducible component $M^{\dagger}_{(F(W),h)}(63\mathbf{v}) \subset M_{(F(W),h)}(63\mathbf{v})$.

\begin{ques}
Are  $M_{\mathrm{LLX}}$  and  $M^{\dagger}_{(F(W),h)}(63\mathbf{v})$ actually isomorphic? If not, can we give a modular description of the singular points? Are they strictly polystable?
\end{ques}

As a first step to understanding the last question, the upcoming paper \cite{BC:Semirigid} shows that the natural morphism
\[
\Sym^n M^{\circ}_{(F(W),h)}(\mathbf{v})\longrightarrow M_{(F(W),h)}(63\mathbf{v})
\]
defines an irreducible component for every $n$.
In particular, if the component $M^{\dagger}_{(F(W),h)}(63\mathbf{v})$ contains semistable sheaves, either it intersects $\Sym^{63} M^{\circ}_{(F(W),h)}(\mathbf{v})$, or some other interesting phenomenon is occurring.

\begin{rema}
An interesting consequence of this fact is that O'Grady's desingularization strategy, which on K3 and abelian surfaces produces OG10 and OG6, cannot be applied to any of the moduli spaces $M_{(F(W),h)}(k\mathbf{v})$ to obtain new hyper-K\"ahler manifolds.
\end{rema}

%%%%%%%%%%%%%%%%%%

\section{Categories}\label{sec:categories}

In this section we indulge more on one of the fundamental questions which are crucial in this survey: assume we have a locally complete family of hyper-K\"ahler manifolds, can we prove that the (possibly very general) member of such a family is isomorphic to a moduli space of stable objects on a different hyper-K\"ahler (or even K-trivial) manifold?

The emphasis in this section is on the idea that such a question can be answered in the positive even assuming that the `hyper-K\"ahler manifold' over which we take stable objects has smaller dimension up to getting to the surface case. Before digging into a more detailed discussion, we should point out that the use of the quotation marks is mainly due to the fact that the surface we want to deal with cannot be a classical commutative one. We know that it is not true that the very general member of a locally complete family of hyper-K\"ahler manifolds of $\mathrm{K}3^{[n]}$-type is isomorphic to a moduli space of stable sheaves on a K3 surface. Indeed, we will soon be dealing with surface-like triangulated categories which will play the role of the noncommutative surfaces we are looking for. This is our starting point.
Back to the main theme of the previous sections, the categorical analogue of Lagrangian submanifolds in such noncommutative setting should be a spherical functor thought as a noncommutative curve.
We will discuss this only in the case of cubic fourfolds.

\subsection{From varieties to (higher) categories}\label{subsect:noncomset}

In order to be able to pass from a commutative variety to a noncommutative one, we should first replace a smooth projective variety $X$ with its bounded derived category of coherent sheaves $\Db(X)$ (or a close relative of it, as we will see). But in order to get some geometric content out of $\Db(X)$ we generally need to endow such a triangulated category with some additional structures. Here we discuss the first one.
A recent overview on the theory with a similar viewpoint is~\cite{P:Survey}.

All schemes in this section are assumed quasi-compact and separated.
We fix a quasi-compact separated scheme $S$, which will be our base scheme.

\subsubsection*{Enhancements}
To make the presentation more intuitive in the beginning of this section, we assume that
$S=\Spec(R)$, where $R$ is a (commutative) ring. 
Then, for an $R$-scheme $\pi\colon X\to S$, the category of \emph{perfect complexes} $\Dperf(X)$ is a triangulated category, linear over $R$.
The first step in this section is to get a grasp on what it means to be linear over $R$ at a categorical level.
To this end, the first observation is that this category admits an \emph{$\infty$-categorical enhancement}.
This means that there exists a stable $\infty$-category $\cD$ whose homotopy category $\Ho\cD$ is equivalent to the given triangulated category. 
We will not need the full extent of the theory, for which we refer to \cite{HTT,HA,SAG} for an extensive presentation, and to~\cite{NCHPD} for a concise overview.
The key ideas to keep in mind are: the homotopy category of a stable $\infty$-category is triangulated and thus it makes sense to look at $\infty$-functors between stable $\infty$-categories whose induced functor at the level of homotopy categories is exact.
Moreover there is an operation of taking a stable $\infty$-category $\cD$ to its closure $\mathrm{Ind}(\cD)$ under small colimits.
This has an inverse operation, consisting in taking the category of \emph{compact objects}, meaning the compact objects in the corresponding homotopy category.\footnote{Recall that an object $C$ in a triangulated category $\cT$ with coproducts is called \emph{compact} if the natural morphism $\bigoplus_{i\in I}\Hom(C,C_i)\to\Hom(C,\bigoplus_{i\in I}C_i)$ is an isomorphism for all $\{C_i\}_{i\in I}\subset\cT$.}

In the case of $\Dperf(X)$, we denote its enhancement by $\Dperfinf(X)$: if $X=\Spec(R)$, such enhancement is naturally associated to thinking to the category of perfect objects as $K^b(R\text{-}\mathrm{proj})$, the homotopy category of bounded complexes of finitely present projective $R$-modules. 
The associated Ind-category is then the stable $\infty$-category $\Dqcinf(X)$, whose homotopy category is the (unbounded) derived category of quasi-coherent sheaves $\Dqc(X)$.
We can then reinterpret the $R$-linearity as follows.
In the case of $\Dperf(X)$, the tensorization by objects in $\Dperfinf(S)$ makes $\Dperfinf(X)$ into a $\Dperfinf(S)$-module in the category of stable $\infty$-categories and this means that $\Dperfinf(X)$ is an \emph{$S$-linear stable $\infty$-category}.
Similarly, in the case of $\Dqc(X)$, we can tensor by objects in $\Dqcinf(S)$, and obtain a \emph{presentable $S$-linear stable $\infty$-category}.
A more down-to-earth understanding about such an $S$-linearization is provided by the fact that any pair of objects $A,B\in\Dqcinf(X)$ come with a mapping object
\[
\mathcal{H}om_S(A,B)\coloneqq\pi_*\mathcal{H}om_X(A,B)\in\Dqc(S),
\]
where $\mathcal{H}om_X(A,B)$ is the derived Hom-sheaf on $X$.

\begin{rema}
In this paper we adopt the language of $\infty$-categories to talk about enhancements of triangulated categories. 
Due to the works \cite{COS1,COS2,Cohn13,Doni}, in the setting we are currently considering, such an approach is equivalent to the one using dg or $A_\infty$ categories. 
Furthermore, the presentation above seems to depend on the choice of special enhancements for the triangulated categories, but this is not the case, in view of \cite[Theorem B]{CNS1} (and its strengthening in \cite{CNS2}).
\end{rema}

What we have just discussed in the very special case of an enhancement for $\Dqc(X)$ can be rephrased for any stable $\infty$-category $\cD$. 
In particular, such a category $\cD$ is $S$-linear if it is a $\Dperfinf(S)$-module. Similarly, in this case, for any pair of objects $A,B\in\cD$, the mapping object $\mathcal{H}om_S(A,B)$ is a complex in $\Dqcinf(S)$. 
Furthermore, given two such stable $\infty$-categories $\cD_1$ and $\cD_2$, we can consider the category
\[
\mathrm{Fun}_{\Dqcinf(S)}\big(\mathrm{Ind}(\cD_1),\mathrm{Ind}(\cD_2)\big),
\]
consisting of exact functors between $\mathrm{Ind}(\cD_1)$ and $\mathrm{Ind}(\cD_2)$ preserving small colimits and the $S$-linear structure. 
Such a category is automatically stable (and presentable $S$-linear) and thus we can look at its compact objects.

The following is crucial for us.

\begin{defi}
Let $\cD$ be an $S$-linear stable $\infty$-category. We say that $\cD$ is:
\begin{enumerate}[{\rm (i)}]
\item {\sf proper} if, for any pair objects $A$ and $B$ in $\cD$, we have $\mathcal{H}om_S(A,B)\in\Dperf(S)$;
\item {\sf smooth} if the identity functor of $\cD$ is a compact object in $\mathrm{Fun}_{\Dqcinf(S)}\big(\mathrm{Ind}(\cD),\mathrm{Ind}(\cD)\big)$.
\end{enumerate}
\end{defi}

The previous definition is motivated by the following.

\begin{exam}
It was observed in \cite[Lemma 4.9]{NCHPD} that if $X$ is a smooth and proper scheme over $S$, then $\Dperfinf(X)$ is smooth and proper. Furthermore, if $X$ and $S$ are noetherian, $\pi\colon X\to S$ is a flat, separated, and of finite type, and $\Dperfinf(X)$ is smooth and proper, then $X$ is smooth and proper over $S$. It should be noted that the proof in \cite{NCHPD} uses a result in \cite{Toen} (later extended in \cite{BZFN}) whose proof (as well as the proof in \cite{BZFN}) contains a gap which is going to be fixed in \cite{COS3}.
Finally, we observe that if $S$ is itself a regular scheme and $X$ is smooth over $S$, then of course $\Dperf(X)=\Db(X)$.
\end{exam}

\subsubsection*{Calabi--Yau categories}

Now we want to introduce some special classes of $S$-linear $\infty$-categories. 
To this extent, recall that for a stable $S$-linear $\infty$-category $\cD$, with identity functor $\id_\cD\in\mathrm{Fun}_{\Dqcinf(S)}\big(\mathrm{Ind}(\cD),\mathrm{Ind}(\cD)\big)$, then
\[
\HH^\bullet(\cD/S)\coloneqq\mathcal{H}om_S(\id_\cD,\id_\cD)\in\Dqc(S)
\]
is the {\sf Hochschild complex} of $\cD$ (over $S$). Consequently, for an integer $i$, the cohomology sheaf
\[
\HH^i(\cD/S):=H^i(\HH^\bullet(\cD/S))
\]
is the {\sf $i$-th Hochschild cohomology} of $\cD$ (over $S$).

\begin{exam}
If $\cD=\Dperfinf(X)$ for $X$ a smooth projective (integral) scheme over $\C$, then we can reconsider the discussion in Section \ref{subsection:defo} so that the Hochschild--Konstant--Rosenberg isomorphism provides the isomorphism 
\[
\HH^i(X):=\HH^i(\Dperfinf(X)/\C)\cong\bigoplus_{p+q=i}H^p\big(X,\bigwedge^q T_X\big).
\]
In particular, $\HH^0(X)\cong\C$ and $\HH^{<0}(X)=0$.
\end{exam}

We say that a stable $S$-linear $\infty$-category has a {\sf Serre functor} if there is an equivalence $S_\cD\colon\cD\xrightarrow{\cong}\cD$ such that the exact functor $S_{\Ho\cD}:=\Ho S_\cD\colon\Ho\cD\to\Ho\cD$ induces natural isomorphisms
\[
\mathcal{H}om_S(A,B)\cong\mathcal{H}om_S(B,S_{\Ho\cD}(A))^\vee,
\]
for any $A,B\in\Ho\cD$.

With these two ingredients in mind we are ready to deal with the following definition.

\begin{defi}
A stable $S$-linear $\infty$-category $\cD$ is an {\sf $n$-Calabi--Yau category} over $S$ (or a CY$n$ category over $S$) if:
\begin{enumerate}[{\rm (i)}]
\item $\cD$ is smooth and proper over $S$,
\item $\cD$ is \emph{connected} over $S$, i.e., $\HH^0(\cD_s/\Spec(k))\cong k$ and $\HH^{< 0}(\cD_s/\Spec(k))=0$,
\item $\cD$ has a Serre functor $S_\cD$ and $S_{\Ho\cD_s}\cong[n]$,
\end{enumerate}
for every point $s=\Spec(k)\hookrightarrow S$.
\end{defi}

We should clarify that $\cD_s$ stands for the base-change of $\cD$ to the point $s$ as defined, for example, in \cite[Section 2.3.1]{NCHPD}. 
In the geometric case, this reduces to the usual Calabi--Yau condition.

\begin{exam}\label{ex:CYeasy}
If $X$ is a smooth projective Calabi--Yau variety of dimension $n$ over a field $k$, then $\Dperfinf(X)$ is a CY$n$ category over $k$. 
Similarly, if $\cX\to S$ is a family of smooth projective Calabi--Yau varieties of dimension $n$, then $\Dperfinf(\cX)$ is a CY$n$ category over $S$.
\end{exam}

\subsection{Examples}

The aim of this section is to provide new examples of Calabi--Yau categories which generalize the geometric one described in Example \ref{ex:CYeasy}. We will mainly follow two paths: the first one consists in recovering a CY$n$ category as an admissible subcategory of a geometric one and the second one boils down to a purely noncommutative deformation of a geometric category. The focus is on the case of case $n=2$ as we want to generalize the example of hyper-K\"ahler manifolds obtained as moduli spaces of stable sheaves in K3 surfaces.

\subsubsection*{Semiorthogonal decompositions and CY$2$ categories}

We begin by recalling that if $\cD$ is a triangulated category, then a {\sf semiorthogonal decomposition}
\begin{equation*}
\cT = \langle \cT_1, \dots, \cT_m \rangle
\end{equation*}
is a sequence of full triangulated subcategories $\cT_1, \dots, \cT_m$ of $\cT$ 
such that: 
\begin{enumerate}[{\rm (i)}]
\item
$\Hom(F, G) = 0$ for all $F \in \cT_i$, $G \in \cT_j$ and $i>j$.
\item
For any $F \in \cT$, there is a sequence of morphisms
\begin{equation*}
0 = F_m \to F_{m-1} \to \cdots \to F_1 \to F_0 = F,
\end{equation*}
such that $\mathrm{cone}(F_i \to F_{i-1}) \in \cT_i$ for $1 \leq i \leq m$. 
\end{enumerate}
In this paper we assume further that the subcategories $\cT_i$ are \emph{admissible}, i.e., their inclusions in $\cT$ have both right and left adjoint.

In the spirit of the discussion of the previous section and for later use, let us point out that the previous definition can be restated for a stable $S$-linear $\infty$-category $\cD$.
In that case, indeed, we require that the triangulated category $\Ho\cD$ admits a semiorthogonal decomposition as above, where the components are of the form $\cT_i=\Ho\cD_i$, where $\cD_i$ are stable $S$-linear $\infty$ subcategories of $\cD$.

\begin{exam}[Cubic fourfolds]\label{ex:cubics}
Smooth cubic fourfolds are parametrized by a $20$-dimensional quasi-projective moduli space. 
Let $\pi\colon\cW\to S$ be a family of cubic fourfolds over a quasi-projective connected complex variety $S$. 
If we denote by $\cO_{\cW}(1)$ the very ample line bundle, we get a semiorthogonal decomposition
\[
\Dperf(\cW)=\left\langle\mathcal{K}u(\cW),\pi^*\Dperf(S),\cO_{\cW}(1)\otimes\pi^*\Dperf(S),\cO_{\cW}(2)\otimes\pi^*\Dperf(S)\right\rangle.
\]
Since, by the discussion in the previous section, $\Db(\cW)$ has an $S$-linear enhancement, the above semiorthogonal decomposition is $S$-linear in the sense discussed above. 
If $S=\Spec(\C)$, then we get a semiorthogonal decomposition $\Db(W)=\left\langle\mathcal{K}u(W),\cO_W,\cO_W(1),\cO_W(2)\right\rangle$, where the triangulated subcategories of $\Db(W)$ generated by $\cO_W(k)$ under extensions, shifts, finite coproducts and direct summands is equivalent to $\Db(\mathrm{pt})$, for $k=0,1,2$.
This explains why $\mathcal{K}u(\cW)$ is indeed the relevant component in the semiorthogonal decomposition. 
We refer to it as the {\sf Kuznetsov component} of the decomposition.

Given any point $s=\Spec(\C)\hookrightarrow S$, one can show that the base change of the Kuznetsov component $\mathcal{K}u(\cW_s)$ is a CY$2$ category (\cite{Kuz:Serre}; see also~\cite[Section 3]{MS:Lectures}). Thus the $S$-linear category $\mathcal{K}u(\cW)$ (actually, its enhancement) provides a $20$-dimensional family of CY$2$ categories, for a suitable choice of $S$. 
Due to the works \cite{AT,BLMNPS:Family} the cubic fourfolds $W$ for which $\mathcal{K}u(W)$ is indeed equivalent to an actual K3 surface are completely classified. Indeed, they form a countable union of divisors in the moduli space of cubic fourfolds.
\end{exam}

\begin{exam}[Gushel--Mukai fourfolds]\label{ex:GM}
The case of Gushel--Mukai fourfolds is similar.
In this case, we have a complete intersection
\[
W\coloneqq\mathrm{Cone}(\mathrm{Gr}(2, 5))\cap Q,
\]
where $\mathrm{Cone}(\mathrm{Gr}(2, 5))\subset\P^{10}$ is the projective cone over the Pl\"ucker embedded Grassmannian $\mathrm{Gr}(2, 5)\subset\P^9$ and $Q\subset\P^{10}$ is a quadric hypersurface in a linear subspace $\P^8\subset\P^{10}$.

To such a $W$ we can associate its Kuznetsov component $\mathcal{K}u(W)\subset\Db(W)$ as in Example~\ref{ex:cubics}. 
Indeed, one consider $\cU_W,\cO_W(1)\in\Db(W)$ to be the pull-backs of the rank-$2$ tautological subbundle and the Pl\"ucker line bundle on $\mathrm{Gr}(2, 5)$. 
With this in mind we get a semiorthogonal decomposition (see~\cite{KP:GM})
\[
\Db(W)=\left\langle\mathcal{K}u(W),\cO_W,\cU_W^\vee,\cO_W(1),\cU_W^\vee(1)\right\rangle.
\]
Similarly to the case of cubic fourfolds, if $\pi\colon\cW\to S$ is a family over a quasi-projective and connected complex variety, then one get an $S$-linear semiorthogonal decomposition of $\Dperf(\cW)$. 
The corresponding Kuznetsov component $\mathcal{K}u(\cW)$ (in fact, as before, its enhancement) is an $S$-linear CY$2$ category. 
Again, the reader may have a look at \cite[Section 3]{MS:Lectures} for a survey on this matter. 
Also in this example, it is actually clear that the very general point in the moduli space of Gushel--Mukai fourfolds is such that the corresponding Kuznetsov component is not geometric, in the sense that it is not equivalent to the derived category of an actual K3 surface.
\end{exam}

As we explained in the previous two examples, the Kuznetsov components arising in the semiorthogonal decompositions of cubic and Gushel--Mukai fourfolds are K3-like noncommutative deformations of actual K3 surfaces. This will be used later to give modular interpretations of the generic members of some families of hyper-K\"ahler manifolds.
Lagrangian submanifolds have categorical counterparts as well.
In fact, by taking smooth hyperplane sections, we can also get corresponding semiorthogonal decompositions of their derived categories and Kuznetsov components.
Notably such Kuznetsov components have homological properties similar to those of a derived category of a curve, and we we will think to them as noncommutative curves in the corresponding noncommutative K3 surface.
To be more precise, in the case of cubic fourfolds (the case of Gushel--Mukai is analogous), we have the following picture.
Let $W$ be a cubic fourfold and let $W_H\hookrightarrow W$ be a smooth cubic threefold hyperplane section of $W$.
We get a semiorthogonal decomposition
\[
\Db(W_H)=\langle\mathcal{K}u(W_H),\cO_{W_H},\cO_{W_H}(1)\rangle,
\]
and the push-forward along the closed embedding induces a functor
\begin{equation}\label{eq:spherical}
\Phi_H\colon \mathcal{K}u(W_H) \longrightarrow\mathcal{K}u(W).    
\end{equation}
The key point is that this functor is \emph{spherical}~\cite{Kuz:CY} with some special properties regarding its associated spherical twists, which makes it the correct analogue of a noncommutative curve.
Varieties of lines can then be interpreted as moduli spaces of stable objects and $\Phi_H$ will indeed induce the inclusion between them (as in Example~\ref{ex:AtomicLagrangians}).

\subsubsection*{Deformations}
We are now ready to discuss yet another way to construct CY$2$ categories over a base. The key difference with respect to Examples~\ref{ex:cubics} and~\ref{ex:GM} is that, while there the construction is purely triangulated here, the higher categorical viewpoint is the key.

\begin{exam}[K3 deformations]\label{ex:K3defo}
In a current work in progress \cite{MPS}, given a family $\cM\to S$ of polarized hyper-K\"ahler manifolds of K3$^{[n]}$-type over a quasi-projective base $S$ such that there is a closed point $s_0\in S$ for which $\cM_{s_0}\cong X^{[n]}$, $X$ a K3 surface, one can construct a relative CY2 category $\cD$ over $S$ with the property that $\Ho\cD_{s_0}\cong\Db(X)$. 
The idea in the construction is to exhibit a higher categorical analogue of the comonad construction in \cite{MM}. 
Such an higher categorical point of view solves the open problem in the latter paper consisting in showing that the comonad construction yields an actual triangulated category. 
While $\cD$ is automatically an $S$-linear proper stable $\infty$-category, its smoothness requires more work.
\end{exam}

\begin{exam}[Abelian surface deformations]\label{ex:abeliandefo}
From the geometric point of view, we know from Section \ref{subsec:Examples} that, given an abelian surface $A$, the crepant resolution of the quotient of $A$ by the group $G:=\langle-1\rangle\cong\Z/2\Z$ provides a K3 surface $X=K_1(A)$. From the homological point of view this is reflected by the fact that $\Db(X)\cong\mathrm{D}^b_G(A)$, where the latter category is the derived category of $G$-equivariant coherent sheaves (see \cite{Inducing} for a quick review). Such an equivalence is a simple instance of the main result in \cite{BKR}. The key observation is that, by \cite{Elagin}, there is a dual $\Z/2\Z$-action on $\Db(X)$ such that its equivariant derived category is precisely $\Db(A)$. Note that making this precise requires enhancing the triangulated categories involved as well as the group actions on them.

The key beautiful idea in \cite{BPPZ} is that this can be make work in families. Let us start with a family $\cX\to S$ of polarized K3 surfaces, where the polarization is with respect to a carefully chosen lattice, such that there is a close point $s_0\in S$ with the property that $X_{s_0}$ carries the $\Z/2\Z$-dual action mentioned above. In \emph{loc.~cit.}, the authors show that, once we enhance the family and the action at the $\infty$-categorical level, such an action deforms all over the family. Thus if one passes to the equivariant category, one gets a CY$2$ category $\cD$ over $S$ such that $\Ho\cD_{s_0}$ is the derived category of an actual abelian surface.
\end{exam}

\subsection{Bridgeland stability conditions}

In this section we want to reconsider Examples \ref{ex:cubics}, \ref{ex:GM}, \ref{ex:K3defo}, and \ref{ex:abeliandefo} and prove that all CY2 categories constructed there carry stability conditions in the sense discussed in Section \ref{subsec:Examples}.
The output will be a solution to the problem we mentioned before: show that the (very) general member of a locally complete family of hyper-K\"ahler manifolds is isomorphic to a moduli space of stable object on a suitable CY2 category.
Each example requires a different approach that we will discuss briefly.

\subsubsection*{Stability conditions on Kuznetsov components}
The key result we want to discuss is actually the following (see~\cite{BLMS,PPZ}):

\begin{theo}\label{thm:stubKuComp}
Let $W$ be either a cubic fourfold or a Gushel--Mukai fourfold. 
Then $\mathcal{K}u(W)$ has a stability condition.
\end{theo}

The result was proven in \cite{BLMS} for cubic fourfold (and also for Fano threefolds). 
The idea is to realize the Kuznetsov component of $W$ as an admissible subcategory of a noncommutative version of $\P^3$ which is obtained in a very geometric way as follows. One picks a line $\ell\subset W$ which is not contained in a projective plane in $W$ (such a line always exists) and to project from it onto a skew $\P^3$. 
The induced conic fibration $\pi_\ell\colon W\dashrightarrow\P^3$ has an homological counterpart as $\mathcal{K}u(W)$ embeds in the derived category of coherent sheaves over the sheaf of the even parts of the Clifford algebra over $\P^3$. 
The general result~\cite[Proposition 5.1]{BLMS} kicks in and we get the desired result. The case of Gushel--Mukai fourfolds is treated in~\cite{PPZ} with a similar but more complicated strategy.

The next observation is that Theorem \ref{thm:stubKuComp} can be reframed in the context of families of stability condition as in \cite{BLMNPS:Family},
%in order to deal with the CY2 categories over an arbitrary base discussed in Examples \ref{ex:cubics} and \ref{ex:GM}. 
in order to study the nonemptiness and the geometric properties of moduli spaces of stable objects. 
As a consequence, we get the following two results, which provide infinitely many locally complete $20$-dimensional new families of hyper-K\"ahler manifolds where every fiber is a moduli space of stable objects in the Kuznetsov component, seen as a `noncommutative K3 surface'.
The cubic fourfolds case is~\cite[Corollary 29.5]{BLMNPS:Family}:

\begin{theo}
For any pair $(a,b)$ of coprime integers, there is a unirational locally complete
$20$-dimensional family, over an open subset of the moduli space of cubic fourfolds, of smooth polarized hyper-K\"ahler manifolds of dimension $2n+2$, where $n=a^2-ab+b^2$. 
The polarization has either degree $6n$ and divisibility $2n$ if $3$ does not divide $n$, or degree and divisibility $\frac{2}{3}n$ otherwise.
\end{theo}

The case of Gushel--Mukai fourfolds is~\cite[Theorem 1.7]{PPZ}:

\begin{theo}[Perry--Pertusi--Zhao]
For any pair $(a,b)$ of coprime integers, there is a unirational locally complete family, over an open subset of the moduli space of Gushel--Mukai fourfolds, of smooth polarized hyper-K\"ahler varieties of dimension $2(a^2+b^2+1)$, degree $2(a^2+b^2)$, and divisibility $a^2+b^2$.
\end{theo}

The easiest examples in both theorems give exactly the hyper-K\"ahler fourfolds we discussed in the previous sections.
For instance, for a cubic fourfold $W$, when $a=b=1$, we get the variety of lines $F(W)$ together with its Pl\"ucker polarization.
The interesting point is that, as mentioned before, stability conditions also exist on the Kuznetsov components of smooth cubic threefolds hyperplane sections $W_H\subset W$ (and they are essentially unique).
Moreover, there is a numerical class for which the moduli space on $\mathcal{K}u(W_H)$ is the variety of lines $F(W_H)$.
The embedding $F(W_H)\subset F(W)$ can then be realized as the image via the functor $\Phi_H$ in~\eqref{eq:spherical} of such stable objects (for a general result, see~\cite[Theorem 8.2]{LLPZ:Higher}).

\subsubsection*{Deforming stability conditions}

In order to deal with the CY2 categories in Example~\ref{ex:K3defo}, we need a general result. 
Let us assume that $\cD$ is a smooth and proper stable $\infty$-category which linear over $\Spec(R)$, where $(R,\fM)$ is a complete DVR. 
We denote the residue field by $k\coloneqq R/\fM$ and the fraction field by $K\coloneqq\mathrm{Frac}(R)$.
By base changing to the closed point and to the generic point, we get the two associated linear stable $\infty$-categories $\cD_k$ and $\cD_K$ which are respectively linear over $k$ and $K$.
The following result will be proved in the work in progress~\cite{LMPSZ1}:

\begin{theo}\label{thm:defo}
If $\cD_k$ has a stability condition $(Z_k,\cA_k)$ with deformable $Z_k$, then there exists a stability condition $(Z_K,\cA_K)$ on $\cD_K$.
\end{theo}

The assumption that $Z_k$ is \emph{deformable}, can be easily formalized in terms of factorization through the relative numerical Grothendieck group defined in \cite{BLMNPS:Family}. 
Since such a condition is satisfied in all the geometric applications that we are about to discuss, we prefer not to delve into more details. 
It should be noted that a key ingredient in the proof of the above result is a well-behaved deformation theory of t-structures, which is indeed dealt with in~\cite{MPS1}.

The point to keep in mind is that, as soon we have a category which fits into a smooth and proper family, then the very general member of such a family has stability conditions as soon as a special member does. 
For instance, this allows to construct examples of stability conditions on very general abelian $n$-folds and hyper-K\"ahler manifolds of K3$^{[n]}$-type and Kum$_n$-type.
This is done in~\cite{LMPSZ2} for the first two cases and in~\cite{Chen} for the last one.
Very recently, stability conditions have been constructed on any projective scheme over an algebraically closed field of characteristic~0 in~\cite{Li:Stability}.

We can apply Theorem~\ref{thm:defo} to the non-geometric setting of Example \ref{ex:K3defo} as well.
In that case, associated to a family $\cM\to S$ of polarized hyper-K\"ahler manifolds with $\cM_{s_0}\cong X^{[n]}$, $X$ a K3 surface, we have a relative CY2 category $\cD$ over $S$ such that $\Ho\cD_{s_0}\cong\Db(X)$. 
Since $\Db(X)$ has stability conditions, we can then apply Theorem~\ref{thm:defo} to get stability condition on $\Ho\cD_{s}$, where $s$ is a very general point in $S$.
By taking relative moduli spaces, in the work in progress~\cite{MPS1} we will prove the following:

\begin{theo}
In the setting above, there exists a countable family of proper closed subsets $\{S_i\}_{i\in\Z}$ such that, for all $s\in S\setminus \left(\bigcup_{i\in\Z}S_i\right)$, $\cM_s$ is isomorphic (as a polarized variety) to a moduli space of stable objects in $\Ho\cD_{s}$.
\end{theo}

One natural question is if we can then use this categorical setting to construct interesting Lagrangian submanifolds on such moduli spaces, by starting from noncommutative curves in the noncommutative K3 surface $\Ho\cD_{s}$, thus generalizing the case of Kuznetsov components discussed in the previous section.

\subsubsection*{Equivariant stability on noncommutative abelian surfaces}
In the same spirit as in the previous case, one can now take the noncommutative deformation $\cD$ of the derived category of an abelian surface constructed in Example~\ref{ex:abeliandefo}. 
Such a deformation is obtained as a suitable equivariant category of the derived category of a family of K3 surfaces. Since we know how to construct stability conditions on K3 surfaces, the very nice idea in~\cite{BPPZ} is to apply an equivariant construction for stability conditions that enhances the one in~\cite{Inducing}. The output is then the following:

\begin{theo}[Bayer--Perry--Pertusi--Zhao]
Let $\cM\to S$ be a family of polarized hyper-K\"ahler manifolds of Kummer-type.
Then there is an open subset $U\subset S$ such that, for every closed point $u\in U$, the variety $\cM_u$ is obtained as the fiber of the Albanese morphism of a moduli space of stable objects on the fiber $\cD_u$ of the $S$-linear CY$2$ category $\cD$ above.
\end{theo}

We refer to \cite{P:Survey} for a more in depth discussion on this result.

%%%%%%%%%%%%%%%%%%

\end{document}